\documentclass[a4paper,11pt,leqno,twoside]{amsart}

\usepackage{amsmath}
\usepackage{amsfonts}
\usepackage{amssymb}
\usepackage{amsthm}
\usepackage{color}
\usepackage{ifpdf}
\usepackage{array}
\usepackage{mathtools}
\usepackage{url}
\usepackage{multirow,bigdelim}
\usepackage{ascmac}
\usepackage{leftidx}
\usepackage{mathrsfs}

\numberwithin{equation}{section}

\newtheorem{theorem}{Theorem}[section]
\newtheorem{proposition}{Proposition}[section] 
\newtheorem{lemma}[proposition]{Lemma}

\newtheorem{remark}[proposition]{Remark}

\newcommand*{\C}{\mathbb{C}}
\newcommand*{\R}{\mathbb{R}}
\newcommand*{\Z}{\mathbb{Z}}

\newcommand{\comment}[1]{}
\title[Chains of RKHS generated by unimodular functions]%
      {Chains of reproducing kernel Hilbert spaces generated by unimodular functions} 
\author[M. Suzuki]{Masatoshi Suzuki}
\AtBeginDocument{%
\begin{abstract}
We present a method to construct a chain of reproducing kernel Hilbert spaces 
controlled by a first-order system of differential equations 
from a given unimodular function satisfying several conditions. 
One of the applications of that method is a conditional but richly general solution 
to the inverse problem of recovering the structure Hamiltonian from a given de Branges space.
\end{abstract}
\maketitle
}

\begin{document}

%
\section{Introduction} 
%

A first-order system of differential equations called a canonical system 
defined by a positive-semidefinite $2 \times 2$ symmetric matrix-valued function $H(t)$ 
gives rise to an entire function $E$ in the Hermite--Biehler class,  
which is a generalization of the exponential functions, see below for details.   
The inverse problem of recovering $H(t)$ from a given function in the Hermite--Biehler class 
is difficult in general, but has been the subject of many studies because of 
its significance and wide applications. 
In that context, naturally, 
the construction of $H(t)$ is discussed on the assumption that 
$E$ belongs to the Hermite--Biehler class. 
However, sometimes we need a way to construct $H(t)$ 
that does not require such an assumption. 

One important example is the entire function $E_\xi(z):=\xi(1/2-iz)+\xi'(1/2-iz)$, 
where $\xi(s):=2^{-1}s(s-1)\pi^{-s/2}\Gamma(s/2)\zeta(s)$ 
for the Riemann zeta-function $\zeta(s)$ 
and the gamma-function $\Gamma(s)$.  
The entire function $E_\xi$ belongs to the Hermite--Biehler class 
if and only if the Riemann hypothesis is true. 
Therefore, if there is a method to construct 
$H(t)$ corresponding to $E_\xi$ 
unconditionally to the Riemann hypothesis, 
it can be applied to the study of the Riemann hypothesis. 
Such a strategy was realized in \cite{Su19_1, Su19_2}, 
resulting in a necessary and sufficient condition 
for the Riemann hypothesis formulated in terms of canonical systems. 

However, the method of \cite{Su19_1} 
is applicable only when the corresponding $H(t)$ is diagonal. 
Thus, for example, it cannot be applied to a Dirichlet $L$-function 
of a non-real Dirichlet character.  
The first purpose of this paper is to solve that problem and 
make it applicable to non-diagonal $H(t)$. 
The second purpose is to extend the range of applications of the theory 
by axiomatically rearranging the method of \cite{Su18, Su19_1, Su19_3}, 
which assumed conditions for concrete integral kernels. 
This makes it possible, for example, to handle 
the examples given in Section \ref{section_3} 
in a unified manner.
For these two purposes, 
we discuss associating a unimodular function 
with a chain of reproducing kernel Hilbert spaces. 
We explain a more specialized and technical outline in the following.
\medskip

A typical source of chains of reproducing kernel Hilbert spaces 
is an entire function $E$ of the Hermite--Biehler class 
$\overline{\mathbb{HB}}$ which consists of 
all entire functions satisfying 
\[
|E^\sharp(z)|<|E(z)|
\] 
in the upper half-plane $\C_+=\{z~|~\Im(z)>0\}$, 
where $F^\sharp(z)=\overline{F(\bar{z})}$. 
We denote by $\mathbb{HB}$ the subspace of $\overline{\mathbb{HB}}$ 
consisting of functions that have no zeros on $\R$. 
First $E$ defines the de Branges space $\mathcal{H}(E)$,  
which is a reproducing kernel Hilbert space consisting of entire functions. 
It is well-known that the set of all de Branges subspaces $\mathcal{H}(E_t)$ 
of $\mathcal{H}(E)$ is totally ordered by set-theoretical inclusion 
and the generators $E_t \in \overline{\mathbb{HB}}$ are controlled by a canonical system, 
which is a system of differential equations of the form 
\begin{equation} \label{c_101}
-\frac{\partial}{\partial t}
\begin{bmatrix} A(t,z) \\ B(t,z) \end{bmatrix}
= 
z 
\begin{bmatrix} 0 & -1 \\ 1 & 0 \end{bmatrix}
H(t)
\begin{bmatrix} A(t,z) \\ B(t,z) \end{bmatrix}
\end{equation}
on an interval $t \in I \subset \R$ parametrized by $z \in \C$, 
where $H(t)$ is a positive-semidefinite $2 \times 2$ symmetric matrix 
for almost all $t \in I$, 
$A(t,z)=(E_t(z)+E_t^\sharp(z))/2$, and $-iB(t,z)=(E_t(z)-E_t^\sharp(z))/2$; 
see Woracek \cite{Wo} for example. 
The matrix-valued function $H$ corresponding to $E$ as above 
is called the {\it structure Hamiltonian} 
of the de Branges space $\mathcal{H}(E)$, 
which is unique up to a reparameterization of $t$ and the normalization $E(0)=1$.  

The inverse problem to recover a structure Hamiltonian from $E$ 
was studied by many authors after the work of de Branges 
(cf. \cite[Section 1]{Su19_1}), 
and recently, a complete characterization 
of structure Hamiltonians of de Branges space 
was obtained by Romanov--Woracek \cite{RomWor19}. 
However, each known method of constructing $H$ 
has its advantages and disadvantages, 
depending on its applications. 
In particular, 
as already mentioned above, 
the method of \cite{Su19_1} 
can be applied only to diagonal Hamiltonians 
that are often referred to as Kre\u{\i}n's strings 
by the correspondence explained after Theorem \ref{thm_2_6}. 
Therefore, it can be applied only 
to the study of so-called self-dual zeta-functions. 

In this paper, the above disadvantage of \cite{Su19_1} is removed 
within a rather broad framework of constructing 
a chain of reproducing kernel Hilbert spaces from 
a function $E$ 
which is not necessarily an entire function. 
Such an extension of the method of constructing $H$ 
would be interesting in its own right. 
Moreover, Hamiltonians that cannot be obtained as structure Hamiltonians 
of de Branges spaces can be systematically obtained from our method. 
For example, let $M(z)$ be a meromorphic function on $\C$ 
having no zeros in $\C_+ \cup \R$, and let us define the spaces
\[
\mathcal{J}_t(M):= e^{izt}M(z)H^2(\C_+) \cap (e^{-izt}M^\sharp(z) H^2(\C_-))
\]
for real numbers $t$, where $\C_-=\{z~|~\Im(z)<0\}$ 
is the lower half-plane 
and $H^2(\C_\pm)$ are Hardy spaces on $\C_\pm$, respectively.   
Then each $\mathcal{J}_t(M)$ is a reproducing kernel Hilbert space 
consisting of meromorphic functions on $\C$. 
If $M$ is an entire function $E \in \overline{\mathbb{HB}}$, 
we have $\mathcal{J}_t(M)=\mathcal{H}(E)$ for $t=0$. 
More generally, $\mathcal{J}_t(M)$ for $t=0$ 
is isomorphic to the model space 
$H^2(\C_+) \ominus (M^\sharp/M)H^2(\C_+)$ 
if $M^\sharp/M$ is an inner function in $\C_+$. 
The model space is isomorphic to a de Branges space 
if $M^\sharp/M$ is a meromorphic inner function in $\C_+$. 
The theory of de Branges spaces and 
model spaces is studied actively by numerous researchers by its importance in connection with various topics of complex and harmonic analysis 
(cf. Garcia--Mashreghi--Ross \cite{MR3526203}, 
and also Chalendar--Fricain--Timotin \cite{MR3589670}, 
Havin--Mashreghi \cite{MR2016246}). 
As detailed in Section \ref{section_2} below, 
it can be seen that the reproducing kernel of $\mathcal{J}_t(M)$ has the form 
\[
J(t;z,w) 
= 
\frac{ 
\overline{A(t,z)}B(t,w)  
- A(t,w)\overline{B(t,z)}
}{\pi(w-\bar{z})} 
\]
under appropriate conditions for $M$. 
Through this formula of the reproducing kernel, we see that 
the reproducing kernel Hilbert space $\mathcal{J}_t(M)$ 
is generally different from de Branges spaces $\mathcal{H}(E)$ 
generated by $E \in \overline{\mathbb{HB}}$ 
and de Branges--Rovnyak spaces $\mathcal{H}(b)$ 
generated by $b \in L^\infty(\R)$. 
On the other hand, the above spaces ordered by inclusion 
$\mathcal{J}_{t}(M) \supset \mathcal{J}_{s}(M)$ 
for $t \leq s$, therefore there exists $t_0 \leq + \infty$ 
such that 
$\mathcal{J}_t(M)\not=\{0\}$ for every $t < t_0$ 
if $\mathcal{J}_t(M)\not=\{0\}$ for some $t < \infty$. 
The chain of spaces $\mathcal{J}_t(M)$, $t<t_0$,  
is controlled by a system of differential equations 
in the sense that there exists a $2 \times 2$ symmetric matrix-valued 
function $H_M(t)$ such that 
the functions $A(t,z)$ and $B(t,z)$ 
in the reproducing kernel $J(t;z,w)$ above satisfy the system \eqref{c_101} 
with $H(t)=H_M(t)$ for $t <t_0 $ and $z \in \C_+$. 
Moreover, we find that $\lim_{t \to t_0}J(t;z,w)=0$ 
if assuming some additional conditions for $M$. 
If $M$ is an entire function belonging to $\overline{\mathbb{HB}}$, 
$\mathcal{J}_{t}(M)=\mathcal{H}(E_t)$ for $0 \leq t <t_0$, 
and $H_M(t)$ is a structure Hamiltonian of $\mathcal{H}(M)$, 
but $H_M(t)$ is generally not a structure Hamiltonian of a de Branges space, 
because $\mathcal{J}_{t}(M)$ is not a de Branges space in general. 

Briefly stated, the method detailed in the next section 
is to define an abstract conjugation on $L^2(\R)$ from a unimodular function on $\R$ 
such that it defines a family of reproducing kernel Hilbert spaces 
as a natural family of conjugation invariant subspaces of $L^2(\R)$. 
The above spaces $\mathcal{J}_{t}(M)$ 
are obtained as Fourier transforms of such invariant spaces.
\medskip

The paper is organized as follows. 
In Section \ref{section_2}, we describe the precise settings and state the main results 
Theorems \ref{thm_2_1}--\ref{thm_2_6}. 
Furthermore, we explain the relationship with Kre\u{\i}n's inverse spectral theory for strings. 
In Section \ref{section_3}, 
we present some non-trivial examples of unimodular functions 
that satisfy all assumptions in the main theorems.
In Section \ref{section_4}, we prove Theorem \ref{thm_2_1}. 
In Section \ref{section_5}, we prove Theorems \ref{thm_2_2}, \ref{thm_2_3}, 
and \ref{thm_2_4}. The most essential new compared to the previous works 
\cite{Su19_1, Su19_3} is the proof of Theorem \ref{thm_2_2}.
In Section \ref{section_6}, we prove Theorem \ref{thm_2_5}. 
Then Theorem \ref{thm_2_6} follows as a corollary.
In Section \ref{section_7}, we describe sufficient conditions for the sixth and the seventh 
of the eight assumptions in Section 2 as a complement to the main results. 

\bigskip

\noindent
{\bf Acknowledgments}~
This work was supported by JSPS KAKENHI Grant Number JP17K05163 and JP23K03050.  
This work was also supported by the Research Institute for Mathematical Sciences, an International Joint Usage/Research Center located in Kyoto University.

%
\section{Results} \label{section_2}
%

In this and subsequent sections, $u$ represents a unimodular function in $L_{\rm loc}^1(\R)$, 
that is, $u$ is a locally integrable function on $\R$ satisfying 
$|u(z)|=1$ for almost every $z \in \R$. 
For technical reasons, we introduce the following conditions 
for unimodular functions $u$ 
and  denote by $U_{\rm loc}^1(\R)$ the set of all $u$ satisfying them: 
\begin{enumerate}
\item[(U1)] the value $u(0)$ is defined, $u(0)\not=0$, and $u$ is H\"{o}lder continuous at $z=0$ 
with exponent $1/2 < \alpha \leq 1$: $|u(z)-u(0)| \ll |z|^\alpha$ as $|z| \to 0$; 
\item[(U2)] there exists a domain $D$ of $\C$, which contains $\R$ 
and is closed under complex conjugation,  
and a meromorphic function $U$ on $D \setminus \R$ 
such that $u(z)$ is the non-tangential limit of $U$ 
at $z$ approaching from both half-planes $\C_+$ and $\C_-$ 
for almost all $z \in \R$. 
Then we often identify $u$ with $U$. 
\end{enumerate}

Condition (U1) means that $u$ is  equal to such a function almost everywhere. 
The reason why the domain $D$ in (U2)  is assumed 
to be symmetric for the real line 
is that if $U$ is a meromorphic function on a domain $D \subset \C_+$, 
then it extends to $\overline{D}\subset \C_-$ by 
$U(z)=1/U^\sharp(z)\,(=1/\overline{U(\bar{z})})$. 
We say that a unimodular function $u \in U_{\rm loc}^1(\R)$ is {\it symmetric} 
if  
\begin{equation} \label{c_201}
u^\sharp(z)=u(-z) \quad \text{for} \quad z \in D. 
\end{equation}
Unimodular functions of the form $u=M^\sharp/M$ 
with a meromorphic function $M$ satisfying 
$M^\sharp(z)=M(-z)$ or $M^\sharp(z)=-M(-z)$ 
are typical examples of symmetric ones.

%
\subsection{Construction of the first-order differential systems} \label{section_2_1}
%

First, we construct systems of differential equations of type \eqref{c_101} 
from unimodular functions satisfying several conditions. 
Let $\mathsf{F}$ and $\mathsf{F}^{-1}$ be the Fourier transform 
and Fourier inverse transform on $L^2(\R)$ respectively:
\[
(\mathsf{F}f)(z)=\int f(x)e^{izx} \, dx, \quad 
(\mathsf{F}^{-1}F)(x)= \frac{1}{2\pi} \int F(z) e^{-izx} \, dz,  
\]
where $\int$ means integration on $\R$ and 
will always be used in this sense. 
Define the operations $\mathsf{J}^\sharp$ and $\mathsf{J}_\sharp$ for functions by 
\[
(\mathsf{J}^\sharp F)(z)=F^\sharp(z):=\overline{F(\bar{z})}, \quad (\mathsf{J}_\sharp f)(x):=\overline{f(-x)}
\]
so that they satisfy the commutative relation 
\[\mathsf{J}^\sharp \mathsf{F}=\mathsf{F}\mathsf{J}_\sharp.\] 
Let $\mathsf{M}_m$ be the operator of multiplication by $m \in L^\infty(\R)$, that is,  
$(\mathsf{M}_mF)(z)=m(z)F(z)$. 
For a unimodular function $u \in L_{\rm loc}^1(\R)$, 
we define the map $\mathsf{K}=\mathsf{K}_u: L^2(\R) \to L^2(\R)$ by 
\begin{equation} \label{c_202}
\mathsf{K} 
= \mathsf{F}^{-1} \mathsf{M}_u \mathsf{J}^\sharp\mathsf{F}
= \mathsf{F}^{-1} \mathsf{M}_u \mathsf{F}\mathsf{J}_\sharp. 
\end{equation}
Then $\mathsf{K}$ is ($\C$-)antilinear (also called conjugate linear), 
that is, 
$\mathsf{K}(af+bg)=\bar{a}\mathsf{K}f+\bar{b}\mathsf{K}g$ 
for $f,g \in L^2(\R)$ and $a,b \in \C$. 
On the properties of antilinear operators, 
see Huhtanen \cite{MR2749452} and Uhlmann \cite{Uhlmann}, 
for example. 
The operator $\mathsf{K}$ satisfies 
$(\mathsf{F}\mathsf{K}f)(z) = u(z) (\mathsf{F}f)^\sharp(z)$ 
for $z \in \R$ by definition. 
Also, $\mathsf{K}$ is isometric, 
because $\mathsf{F}$ is isometric up to scaling, 
$\mathsf{J}^\sharp$ and $\mathsf{J}_\sharp$ are clearly isometric, 
and  $\mathsf{M}_u$ is isometric for a unimodular function $u$. 
Further $\mathsf{K}^2=1$, since 
$
\mathsf{F}\mathsf{K}^2f(z)
 =
(\mathsf{F}\mathsf{K}(\mathsf{K}f))(z) 
= u(z) (\mathsf{F}\mathsf{K}f)^\sharp(z)
= u(z)u^\sharp(z)(\mathsf{F}f)(z)
$
and $u(z)u^\sharp(z)= \vert u(z) \vert^2 =1$ for $z \in \R$. 
We summarize the above properties of $\mathsf{K}$ as follows 
recalling that, for an antilinear operator $\mathsf{T}$,  
the adjoint $\mathsf{T}^\ast$ is defined by
$
\langle \mathsf{T}f, g \rangle
= \overline{\langle f, \mathsf{T}^\ast g \rangle} 
= \langle \mathsf{T}^\ast g, f \rangle
$, 
where 
$\langle f,g \rangle = \int f(x)\overline{g(x)}\,dx$ 
for $f, g \in L^2(\R)$. 

\begin{proposition} \label{prop_2_1}
For a unimodular function $u$ in $L_{\rm loc}^1(\R)$, 
the map $\mathsf{K}=\mathsf{K}_u$ is an antilinear isometric involution on $L^2(\R)$, 
in other words, $\mathsf{K}$ is an abstract conjugation on $L^2(\R)$. 
Hence, in particular, $\mathsf{K}$ is self-adjoint: $\mathsf{K}=\mathsf{K}^\ast$. 
\end{proposition}

For some special unimodular function $u$, 
the operator $\mathsf{K}$ is represented as an integral operator with a continuous integral kernel 
(see \cite[Theorem 2.1]{Su19_2}, for example). 
If we allow the integral kernel to be a tempered distribution, 
$\mathsf{K}$ is always an integral operator as follows.  
Every $u$ in $L_{\rm loc}^1(\R)$ can be regarded as a tempered distribution on $\R$ by 
$(u,g) = \int u(x)g(x) dx$, $g \in S(\R)$, 
where $S(\R)$ is the Schwartz space. 
Therefore,  
there exists a tempered distribution $k$ on $\R$ such that  
$u = \mathsf{F}k$, 
since the Fourier transform $\mathsf{F}$ extends to the space of tempered distribution $S'(\R)$ 
as a bijection 
(so $k=\mathsf{F}^{-1}u$). 
Then the operator $\mathsf{K}$ of \eqref{c_202} is expressed as the integral operator 
\[
(\mathsf{K}f)(x)=  (k\ast \mathsf{J}_\sharp f)(x)
= \int k(x+y)\overline{f(y)}\,dy 
\]
by the product rule of the Fourier transform, 
where $\ast$ stands for the additive convolution. 
In some cases, $k$ can be regarded as a function, but it never belongs to $L^1(\R)$ 
by the Riemann--Lebesgue theorem. 
The tempered distribution $k=\mathsf{F}^{-1}u$ is real-valued 
if and only if $u$ is symmetric, that is, $u$ satisfies \eqref{c_201}. 

For $t \in \R$, we define the compression 
$\mathsf{K}[t]: L^2(-\infty,t) ~\to~ L^2(-\infty,t)$ 
of 
$\mathsf{K}$ by 
\[
\mathsf{K}[t]:=\mathsf{P}_t\mathsf{K}\vert_{L^2(-\infty,t)}, 
\]
where $\mathsf{P}_t$ is the orthogonal projection 
from $L^2(\R)$ to $L^2(-\infty,t)$. 
Then, $\mathsf{K}[t]=\mathsf{P}_t\mathsf{K}\mathsf{P}_t$ on $L^2(-\infty,t)$. 
Since $\mathsf{K}$ is isometric on $L^2(\R)$, 
the inequality of operator norm $\Vert \mathsf{K}[t] \Vert_{\rm op} \leq 1$ always holds.  
Now, we introduce the following condition on $u$ in $L_{\rm loc}^1(\R)$: 
\begin{enumerate}
\item[(O1)] $\Vert \mathsf{K}[t] \Vert_{\rm op} <1$ for some $t \in \R$, 
\end{enumerate} 
where $\Vert \cdot \Vert_{\rm op}$ is the operator norm for operators on $L^2(-\infty,t)$. 
Note that if $\Vert \mathsf{K}[t] \Vert_{\rm op} <1$ for one $t$  
then it holds for all smaller $t$'s, 
since $\Vert \mathsf{K}[s]\Vert_{\rm op} \leq \Vert \mathsf{K}[t]\Vert_{\rm op}$ 
for $s <t$ by definition of the operator norm. 

Henceforth, we suppose that $u$ belongs to the subspace 
$U_{\rm loc}^{1}(\R)$ $(\subset L_{\rm loc}^1(\R))$ 
in order that $\mathsf{K}[t]1$ is defined as a function belonging to $L^2(-\infty,t)$ 
(cf. Proposition \ref{prop_4_1} below) and for other technical reasons.
If $\Vert \mathsf{K}[t] \Vert_{\rm op} <1$, the equations 
$(1 + \mathsf{K}[t])\varphi= -\mathsf{K}[t]1$ 
and 
$(1 - \mathsf{K}[t])\psi= \mathsf{K}[t]1$ 
for $\varphi$, $\psi \in L^2(-\infty,t)$ 
have unique solutions. 
Using the solutions $\varphi$ and $\psi$, 
we define the functions $\Phi$ and $\Psi$ on $\R$ 
by 
$\Phi := 1- \mathsf{K}(\varphi + \mathsf{P}_t1)$
and 
$\Psi := 1+ \mathsf{K}(\psi + \mathsf{P}_t1)$. 
Then they solve equations 
\begin{equation} \label{c_203}
\Phi + \mathsf{K}\mathsf{P}_t \Phi= 1,  
\end{equation}
\begin{equation} \label{c_204}
\Psi - \mathsf{K}\mathsf{P}_t \Psi = 1. 
\end{equation}
We find that the solutions $\Phi$ and $\Psi$ 
satisfy a certain system of partial differential equations 
if assuming the following technical conditions (Proposition \ref{prop_4_5} below). 
To state such conditions, 
we extend the action of $\mathsf{K}$ to 
the space of tempered distributions $S'(\R)$ 
by using \eqref{c_202} (see Section \ref{section_4_1} for details). 
Then the conditions are stated as follows:  
\begin{enumerate}
\item[(O2)] 
For the above solutions $\Phi(t,x)$ and $\Psi(t,x)$ of \eqref{c_203} and \eqref{c_204}, 
derivatives 
$(\partial/\partial t)\Phi$, $(\partial/\partial t)\Psi$, 
$(\partial/\partial t)\mathsf{F}\Phi$, $(\partial/\partial t)\mathsf{F}\Psi$ 
with respect to $t$ 
are defined as tempered distributions of $x$, 
and the commutativity 
\[
\frac{\partial}{\partial t}\mathsf{F} \Phi
= \mathsf{F}\frac{\partial}{\partial t} \Phi, 
\qquad 
\frac{\partial}{\partial t}\mathsf{F} \Psi
= \mathsf{F}\frac{\partial}{\partial t}\Psi,
\]
hold, respectively, whenever $\Vert \mathsf{K}[t] \Vert_{\rm op}<1$; 
\item[(O3)] The values of $\Phi(t,x)$ and $\Psi(t,x)$  at $x=t$ are well-defined 
and nonzero, whenever $\Vert \mathsf{K}[t] \Vert_{\rm op}<1$;  
\item[(O4)] $\mathsf{P}_t\mathsf{K}\delta_t$ is defined as a function belonging to $L^2(-\infty,t)$, 
or else the kernels of $(1 \pm \mathsf{K}[t]): \mathsf{P}_tS'(\R) \to \mathsf{P}_tS'(\R)$ 
are zero, whenever $\Vert \mathsf{K}[t] \Vert_{\rm op}<1$, 
\end{enumerate}
where $\delta_t(x)=\delta(x-t)$ for the Dirac distribution $\delta$ at the origin, 
$(\mathsf{P}_tf)(x) = \mathbf{1}_{(-\infty,0)}(x-t) f(x)$ for tempered distributions $f$,  
and the kernels of $(1 \pm \mathsf{K}[t])$ are considered as $\R$-linear maps. 
Condition (O4) guarantees that the solutions of 
\eqref{c_203} and \eqref{c_204} are unique in $S'(\R)$.

Now we introduce two functions $\tilde{A}(t,z)$ and $\tilde{B}(t,z)$  by
\begin{equation} \label{c_205}
\tilde{A}(t,z) := -\frac{iz}{2} (\mathsf{F}(1-\mathsf{P}_t)\Psi)(z), 
\quad 
-i\tilde{B}(t,z) := -\frac{iz}{2} (\mathsf{F}(1-\mathsf{P}_t)\Phi)(z), 
\end{equation}
where the Fourier transforms are taken as tempered distributions. 
They play a central role in all of the following results in this section. 

\begin{theorem} \label{thm_2_1} 
Let $u \in U_{\rm loc}^1(\R)$ 
and define $\mathsf{K}=\mathsf{K}_u$ as above.  
Suppose that conditions (O1), (O2), (O3), (O4) are satisfied, 
and let $T$ be a real number such that $\Vert \mathsf{K}[t] \Vert_{\rm op} <1$ 
for all $t<T$. Then 
\begin{enumerate} 
\item[(1)] 
$\tilde{A}(t,z)$ and $\tilde{B}(t,z)$ 
are defined by \eqref{c_205} 
as tempered distributions on $\R$ for each fixed $t<T$; 
\item[(2)]  
$\tilde{A}(t,z)$ and $\tilde{B}(t,z)$ extend to meromorphic functions on $(\C_+ \cup D)\setminus \R$, 
and they are holomorphic on $\C_+$ for each fixed $t<T$, where $D$ is the domain for $u$ in (U2); 
\item[(3)]  the limit equations 
$\lim_{z \to x}\tilde{A}(t,z)=\tilde{A}(t,x)$ 
and 
$\lim_{z \to x}\tilde{B}(t,z)=\tilde{B}(t,x)$ 
hold for almost all $x \in \R$ 
if $z$ tends to $x$ non-tangentially either in $\C_+$ or $\C_-$; 
\item[(4)]  
$\tilde{A}(t,z)$ and $\tilde{B}(t,z)$ 
 satisfy the functional equations 
\begin{equation} \label{c_206}
\tilde{A}(t,z) = u(z) \tilde{A}^\sharp(t,z), 
\quad 
\tilde{B}(t,z) = u(z) \tilde{B}^\sharp(t,z)
\end{equation}
for $z \in D$; 
\item[(5)]  $\tilde{A}(t,z)$ and $\tilde{B}(t,z)$ are continuous with respect to $t$ 
for fixed $z \in \C_+ \cup D$ except for their (isolated) singularities; 
\item[(6)]  $\tilde{A}(t,z)$ and $\tilde{B}(t,z)$  satisfy the first order system 
\begin{equation} \label{c_207} 
-\frac{\partial}{\partial t}
\begin{bmatrix}
\tilde{A}(t,z)\\ \tilde{B}(t,z)
\end{bmatrix}
= z 
\begin{bmatrix}
0 & -1 \\ 1 & 0
\end{bmatrix}
H(t)
\begin{bmatrix}
\tilde{A}(t,z)\\ \tilde{B}(t,z)
\end{bmatrix}
\end{equation}
for $t< T$ and $z \in \C_+ \cup D$, 
where $H(t)$ is the matrix-valued function defined by 
\begin{equation} \label{c_208} 
H(t)=H_u(t)
:=
\begin{bmatrix}
\alpha(t) & \beta(t) \\ \beta(t) & \gamma(t)
\end{bmatrix}
\end{equation}
and 
\begin{equation} \label{c_209}
\aligned 
\alpha(t) 
&= \frac{|\Phi(t,t)|^2}{\Re(\Phi(t,t)\overline{\Psi(t,t)})}
= \frac{1}{\Re(\Psi(t,t)/\Phi(t,t))}, \\
\beta(t) 
&= \frac {\Im(\Phi(t,t)\overline{\Psi(t,t)})}{\Re(\Phi(t,t)\overline{\Psi(t,t)})}
= \frac {\Im(\Phi(t,t)/\Psi(t,t))}{\Re(\Phi(t,t)/\Psi(t,t))}
= -\frac {\Im(\Psi(t,t)/\Phi(t,t))}{\Re(\Psi(t,t)/\Phi(t,t))}, \\
\gamma(t) 
&= \frac{|\Psi(t,t)|^2}{\Re(\Phi(t,t)\overline{\Psi(t,t)})}
= \frac{1}{\Re(\Phi(t,t)/\Psi(t,t))}. 
\endaligned 
\end{equation}
\item[(7)]  $H(t)$ of (6) 
belongs to ${\rm SL}_2(\R) \cap {\rm Sym}_2(\R)$ for all $t<T$. 
If $u$ is symmetric, $H(t)$ is diagonal. 
\end{enumerate}
\end{theorem}

Notation $\alpha(t)$, $\beta(t)$, $\gamma(t)$ in \eqref{c_208} and \eqref{c_209} 
correspond to $\alpha'(t)$, $\beta'(t)$, $\gamma'(t)$ in  de Branges' book \cite{MR0229011}. 
The first-order differential system in (6) has a different range of complex parameter $z$ 
from usual theory of canonical systems. Such systems are called 
{\it lacunary} canonical systems in \cite{Su19_3}.  

%
\subsection{Chains of reproducing kernel Hilbert spaces} \label{section_2_2}
%

Second, we describe that 
the system of differential equations \eqref{c_207} 
controls the structure of a chain of reproducing kernel Hilbert spaces  
if assuming further conditions for a given unimodular function.

For $\mathsf{K}=\mathsf{K}_u$ and each $t \in \R$, we denote by $\mathcal{V}_t(u)$ 
the space of all functions $f \in L^2(\R)$ 
such that both $f$ and $\mathsf{K}f$ have their supports in $[t, \infty)$: 
\[
\mathcal{V}_t(u) = L^2(t,\infty) \cap \mathsf{K} L^2(t,\infty). 
\]
By definition, $\mathcal{V}_t(u)$ is a conjugation-invariant subspace of $L^2(\R)$ 
with respect to the conjugation $\mathsf{K}$. 
We do not need any of the conditions (O1)--(O4) to define $\mathcal{V}_t(u)$, 
but the following condition is necessary 
for the discussion about $\mathcal{V}_t(u)$ to be meaningful: 
\begin{enumerate}
\item[(O5)] $\mathcal{V}_t(u)$ is non-zero for some $t \in \R$. 
\end{enumerate}

If $\mathcal{V}_t(u)$ is non-zero, 
 $\mathsf{F}(\mathcal{V}_t(u))$ is a reproducing kernel Hilbert space 
consisting of functions on $\C_+ \cup D$ 
that are holomorphic on $\C_+$ and meromorphic on $D \cap \C_-$. 
The reproducing kernel $j(t; z, w)$ of $\mathsf{F}(\mathcal{V}_t(u))$ 
is expressed as $j(t;z,w) = \frac{1}{2\pi} \langle Y_w^t, Y_z^t \rangle$ 
using the vector $Y_z^t \in \mathcal{V}_t(u)$ 
satisfying $\langle f, \overline{Y_z^t} \rangle=\mathsf{F}f(z)$ 
for any $f \in \mathcal{V}_t(u)$ (see Section \ref{section_5} below for details). 
More specifically, $j(t;z,w)$ has the following explicit formula  
consisting of functions defined in \eqref{c_205}. 
\begin{theorem} \label{thm_2_2}
Let $u\in U_{\rm loc}^1(\R)$. 
Suppose that conditions (O1), (O2), (O3), (O4), (O5) are satisfied, 
and let $t \in \R$ such that 
$\Vert \mathsf{K}[t] \Vert_{\rm op}<1$ 
and $\mathcal{V}_t(u)\not=0$.  
Let $j(t;z,w)$ be the reproducing kernel of  $\mathsf{F}(\mathcal{V}_t(u))$. 
Then,  
\begin{equation} \label{c_210} 
j(t;z,w) = \frac{1}{2\pi} \langle Y_w^t, Y_z^t \rangle
= 
\frac{ 
\overline{\tilde{A}(t,z)}\tilde{B}(t,w)  
- \tilde{A}(t,w)\overline{\tilde{B}(t,z)}
}{\pi(w-\bar{z})} 
\end{equation}
holds for $z,w \in \C_+ \cup D$, where $D$ is the domain for $u$ in (U2).  
\end{theorem}
\begin{remark} 
A formula similar to \eqref{c_210} is proved in \cite[Section 4.2]{Su19_1}, 
but the assumed set of conditions in Theorem \ref{thm_2_2} 
is quite different from that in \cite{Su19_1}. 
See also the comments after Theorem \ref{thm_2_6}.
\end{remark}

As is clear from the definition, 
the spaces $\{\mathcal{V}_t(u)\}_{t \in \R}$ are totally ordered by set-theoretical inclusion 
$\mathcal{V}_s(u) \subset \mathcal{V}_t(u)$ for $t<s$ 
and the inclusion is an isometric embedding as a Hilbert space. 
However, we should note that, unlike de Branges spaces, 
not necessarily all $\mathsf{K}$-invariant subspaces of $\mathcal{V}_t(u)$ 
have the shape of $\mathcal{V}_s(u)$ ($t<s$), and therefore, 
not necessarily the set of all $\mathsf{K}$-invariant subspaces of $\mathcal{V}_t(u)$ is totally ordered.

If $\mathcal{V}_t(u) = \{0\}$ for some $t \in \R$, 
then $\mathcal{V}_s(u) =\{0\}$ for all $s \geq t$ by definition. 
Therefore, it makes sense to consider the value 
\[
t_0=t_0(u):=\sup\{ t \in \R \,|\, \mathcal{V}_t(u) \not=\{0\}\}. 
\]
Under condition (O5), $t_0$ is determined as a finite real number or $+\infty$. 
To use the above results in the study 
of the chain of spaces $\{\mathsf{F}(\mathcal{V}_t(u))\}_{t<t_0}$, 
we introduce the following conditions: 
\begin{enumerate}
\item[(O6)]  $\Vert \mathsf{K}_u[t] \Vert_{\rm op} < 1$ for every $t<t_0$;  
\end{enumerate}
\begin{enumerate}
\item[(O7)]  $\Re(\Phi(t,t)\overline{\Psi(t,t)})>0$ for every $t < t_0$. 
\end{enumerate}
Note that (O1) is automatically satisfied assuming (O5) and (O6). 

Let $H^\infty=H^\infty(\C_+)$ be the space of all bounded analytic functions in $\C_+$. 
A function $\theta \in H^\infty$  
is called an {\it inner function} in $\C_+$ 
if $\lim_{y \to 0+}|\theta(x+iy)|=1$ for almost all $x \in \R$ 
with respect to the Lebesgue measure. 
An inner function $\theta$ 
defines a measurable unimodular function on $\R$ 
by taking a nontangential limit at a point of $\R$ 
and extends to the lower half-plane 
by setting $\theta(z):=1/\theta^\sharp (z)$ for $z \in \C_-$, 
in particular, $\theta$ is meromorphic on $\C\setminus \R$. 
If an inner function $\theta$ in $\C_+$ extends to a meromorphic function on $\C$, 
it is called a {\it meromorphic inner function} in $\C_+$. 
For an inner function $\theta$, the space  
\[
\mathcal{K}(\theta) = H^2(\C_+) \ominus \theta H^2(\C_+)
\]
defined as an orthogonal complement is called a {\it model subspace}. 

\begin{theorem} \label{thm_2_3}
Let $u \in U_{\rm loc}^1(\R)$. 
Suppose that conditions (O2), (O3), (O4), (O5), (O6), (O7) are satisfied. 
Then the function defined by 
\begin{equation} \label{c_211}
\theta(t,z):=\frac{\tilde{A}(t,z) + i \tilde{B}(t,z)}{\tilde{A}(t,z) - i \tilde{B}(t,z)}
\end{equation}
is an inner function in $\C_+$ and 
$\mathsf{F}(\mathcal{V}_t(u)) = \mathcal{K}(\theta(t,z))$ 
for every $t < t_0$. 
\end{theorem}

Because of the connection with the theory of de Branges spaces, 
we are particularly interested in unimodular functions of the form 
$u(z)=M^\sharp(z)/M(z)$ for some meromorphic function $M(z)$ on $\C$. 
Using the functions $\tilde{A}(t,z)$ and $\tilde{B}(t,z)$ in \eqref{c_205}, we define 
\begin{equation} \label{c_212}
\aligned 
A(t,z) := & M(z)\tilde{A}(t,z), \quad B(t,z) := M(z)\tilde{B}(t,z),  \\
& E(t,z) :=A(t,z) -iB(t,z). 
\endaligned 
\end{equation}
With the assumptions in Theorem \ref{thm_2_3}, 
$\Theta(t,z):=E^\sharp(t,z)/E(t,z)$ extends to a meromorphic inner function 
for every $t<t_0$. 
In general, if $\Theta$ is a meromorphic inner function, 
there exists $E \in \mathbb{HB}$ 
such that $\Theta=E^\sharp/E$ 
and the model subspace $\mathcal{K}(\Theta)$ 
is isometrically isomorphic to the de Branges space $\mathcal{H}(E)$ 
by the map $F \mapsto EF$ (\cite[\S 2.3 and \S2.4]{MR2016246}), 
where $\mathbb{HB}$ is the subspace consisting of 
functions that have no zeros on $\R$ as before. 
Therefore, 
if $\Theta(\tau,z)$ is a meromorphic inner function, 
it is expected that $H(t)$ in Theorem \ref{thm_2_1} 
is nothing but the structure Hamiltonian of the de Branges space $\mathcal{H}(E(\tau,z))$. 
To realize this expectation, we introduce one more condition: 
\begin{enumerate}
\item[(O8)]  $\mathcal{V}_{t_0}(u)=\{0\}$ if $t_0 <\infty$. 
\end{enumerate}
Note that, when $t_0<\infty$, 
both $\mathcal{V}_{t_0}(u)=\{0\}$ and $\mathcal{V}_{t_0}(u)\not=\{0\}$ can occur. 
See examples in Section \ref{section_3}.  

\begin{theorem} \label{thm_2_4} 
Let $u\in U_{\rm loc}^1(\R)$. 
Suppose that $u=M^\sharp/M$ for some meromorphic function $M$ on $\C$ 
that is holomorphic on $\C_+ \cup \R$ and has no zeros in $\C_+$. 
Suppose that (O2), (O3), (O4), (O5), (O6), (O7), (O8) are satisfied. 
Define $E(t,z)$ by \eqref{c_205} and \eqref{c_212} for $t<t_0$. 
Then, for any $t_1<t_0$, 
$E(t_1,z)$ 
 is an entire function of $\overline{\mathbb{HB}}$ 
and $H(t)$ on $[t_1,t_0) $ 
defined by \eqref{c_208} and \eqref{c_209} is the structure Hamiltonian 
of the de Branges space $\mathcal{H}(E(t_1,z))$.  
\end{theorem}

The asymptotic behavior of the reproducing kernel of $\mathcal{H}(E(t,z))$ as $t \to t_0$ 
is clarified in the proof of the theorem, but the asymptotic behavior of $E(t,z)$ as $t \to t_0$ 
is more difficult and will not be studied in this paper. 
See \cite[Theorem 41]{MR0229011} and \cite[Theorems 1.34, 3.15, 4.20]{Linghu15} 
for results on the asymptotic behavior of $E(t,z)$. 

%
\subsection{Specialization to de Branges spaces} \label{section_2_3}
%

Further specializes situation in the previous subsection. 
In Theorem \ref{thm_2_4}, 
there is no direct relation between $u(z)$ and $E(t,z)$, 
but if $u$ is a restriction of a meromorphic inner function 
$\Theta=E^\sharp/E$, then the chain of spaces $\mathsf{F}(\mathcal{V}_t(u))$ for $t \geq 0$ 
is isomorphic to the chain of de Branges subspaces of the de Branges space $\mathcal{H}(E)$. 
We state it after giving a result on $u=\theta$, 
which is an inner function but not necessarily a meromorphic inner function.

\begin{theorem} \label{thm_2_5}
Suppose that $u \in U_{\rm loc}^1(\R)$ 
is a restriction of an inner function $\theta$ in $\C_+$ 
and that (O1), (O2), (O3), (O4) are satisfied. 
Then 
\begin{equation} \label{c_213}
\tilde{A}(0,z) = \frac{1}{2}(1+\theta(z)), 
\qquad 
-i\tilde{B}(0,z) = \frac{1}{2}(1-\theta(z)) 
\end{equation}
for $z \in \C$. 
In particular, $\mathsf{F}(\mathcal{V}_0(u)) = \mathcal{K}(\theta)$. 
\end{theorem}

From this result, it can be understood that 
Theorem \ref{thm_2_1} 
solves the direct problem for the lacunary canonical system on 
$[0,t_0)$ associated with 
the particular Hamiltonian 
$H(t)$ defined by \eqref{c_208} and \eqref{c_209} 
and equality \eqref {c_213} as the initial condition at zero 
by providing the explicit solution $(\tilde{A}(t,z), \tilde{B}(t,z))$.

If $u$ is a restriction of an inner function, $\mathsf{K}[t]=0$ for nonpositive $t$ 
by Proposition \ref{prop_6_2} below. 
Further, we have $\Phi(t,t)=\Psi(t,t)=1$ if $\Phi(t,x)$ and $\Psi(t,x)$ are continuous at $x=t$, 
and therefore $H(t)$ is the identity matrix for nonpositive $t$. 
In this sense, the nontrivial range of $t$ for $u=\theta$ is $0<t <t_0$, 
and it is actually meaningful as follows. 

For a meromorphic inner function 
$\Theta=E^\sharp/E$ with $E \in \mathbb{HB}$, 
we define $A(t,z)$ and $B(t,z)$ by \eqref{c_205} and \eqref{c_212}  with $M=E$. 
Then $E(0,z)=E(z)$ by Theorem \ref{thm_2_5}. 
Therefore, as a corollary of Theorem \ref{thm_2_4}, 
we obtain the following result 
which solves the inverse problem of finding the structure Hamiltonian 
from a given generator $E$ of a de Branges space 
(see also Proposition \ref{prop_6_2}): 

\begin{theorem} \label{thm_2_6}
Let $\Theta$ be a meromorphic inner function such that $\Theta=E^\sharp/E$ 
for some $E \in \mathbb{HB}$. Define $\mathsf{K}$ for $u=\Theta$ 
and suppose that (O2), (O3), (O4), (O6), (O7), (O8) are satisfied. 
Then $H(t)$ on $[0,t_0) $ 
defined by \eqref{c_208} and \eqref{c_209} is the structure Hamiltonian 
of the de Branges space $\mathcal{H}(E)$. 
\end{theorem}

Note that there are many de Branges spaces to which 
Theorem \ref{thm_2_6} does not apply, 
as in the example in Section \ref{section_3_2}.

\subsection{Comparison with Kre\u{\i}n's method} 
Here, we clarify the relationship between the method proposed 
in the present paper and 
Kre\u{\i}n's inverse spectral theory of strings on a half-line 
(\cite{Krein54_1, Krein54_2}, see also \cite{Kats94}). 
For details on the relationship between the spectral theory of strings, canonical spaces, 
and de Branges spaces, see Dym--McKean \cite{DymMcO76} and Langer--Winkler \cite{LanWin98}. 

Let $\Theta=E^\sharp/E$ be as in Theorem \ref{thm_2_6}. 
If $\Theta$ is symmetric, then 
the structure Hamiltonian $H$ of the de Branges space $\mathcal{H}(E)$ is diagonal 
and regular (limit circle) at $t=0$. 
Such a Hamiltonian is associated with 
a Kre\u{\i}n's string $S[m,L]$ consisting with its length $L$ ($0< L \leq \infty$) 
and a nondecreasing right-continuous function $m(x)$ defined on $[0,L)$, 
the mass distribution. 
From a diagonal Hamiltonian $H(t)={\rm diag}(1/\gamma(t),\gamma(t))$ on $[0,t_0)$, 
a string $S[m,L]$ is obtained by defining 
\[
x=f(t):=\int_{0}^{t} \frac{1}{\gamma(s)} \,ds, \quad 
L:= \int_{0}^{t_0} \frac{1}{\gamma(s)} \,ds, \quad 
m(x):=\int_{0}^{f^{-1}(x)} \gamma(s)\,ds. 
\]
From the obtained string, the original $H(t)$ is restored by defining 
\[
g(x):=\int_{0}^{x} \sqrt{m'(y)} \, dy, \quad 
\gamma(t):=\sqrt{m'(g^{-1}(t))}.  
\]
This correspondence is a bit more general, but we have described it in a limited situation for simplicity. 
The Titchmarsh--Weyl function $Q_S$ 
of the string  $S[m,L]$ 
is related to the Titchmarsh--Weyl function $Q_H$ of $H(t)$ 
as 
\[
z \,Q_S(z^2) = Q_H(z)\,\left(:= i \,\frac{1-\Theta(z)}{1+\Theta(z)}\right), 
\]
and the former admits the representation
\[
Q_S(z) = b + \int_{0}^{\infty} \frac{d\tau(\lambda)}{\lambda-z}, 
\]
where $b$ is a nonnegative constant and $\tau$ is a measure on $[0,\infty)$ 
called the principal spectral measure of the string $S[m, L]$.

In \cite{Krein54_1, Krein54_2}, 
Kre\u{\i}n announced the method to recover the string from the principal spectral measure. 
In the following outline, $b=0$ is assumed, 
and all assumptions on the principal spectral measure $\tau$ (or its transfer function) are omitted. 
First, we introduce the transfer function 
\begin{equation} \label{eq_20206_1}
\Phi(t)
= \int_{0}^{\infty} \frac{1-\cos(t\sqrt{\lambda})}{\lambda} \, d\tau(\lambda)
\end{equation}
on $[0,\infty)$, then consider the family of integral equations  
\begin{equation} \label{eq_20206_2}
2\Phi'(0)q(x) + \int_{-t}^{t} \Phi''(x-y) q(y) \, dy=1, \quad 0 \leq t \leq t_0. 
\end{equation}
Under appropriate conditions for $\Phi(t)$, 
this equation has unique integrable solution $x \mapsto q(x;t)$ on $[-t,t]$ 
for each $0 \leq t \leq t_0$. 
Using the solution $q(x;t)$, we set
\begin{equation} \label{eq_20206_3}
p(t) := \frac{d}{dt} \int_{-t}^{t} q(x;t) \, dx = \frac{2q(t,t)^2}{q(0,0)}.
\end{equation}
Then, we get 
$m(f(t))= \int_{0}^{t} p(s) \, ds$ with 
$f(t)  = \int_{0}^{t}1/p(s) \, ds$. 
This implies 
\begin{equation} \label{eq_20206_4}
\gamma(t) = p(t).
\end{equation}
In this way, $H(t)$ is restored from the spectral measure of a given string. 
The similarity between \eqref{eq_20206_2}, \eqref{eq_20206_3}, \eqref{eq_20206_4} 
and \eqref{c_203}, \eqref{c_204}, \eqref{c_209} 
(or, the more direct matches are (1.5), (1.6), (1.10) of \cite{Su19_3}) is remarkable. 
It is even more striking if comparing 
the Fourier transforms  
\[
\int_{0}^{\infty} \Phi(t) e^{izt} \, dt 
= -\frac{1}{z^2}\left(\frac{1-\Theta(z)}{1+\Theta(z)}\right) \quad  
\text{for large} \quad \Im z > 0
\]
and $(\mathsf{F}k)(z)=\Theta(z)$ for the kernel $k=\mathsf{F}^{-1}\Theta$ of the operator $\mathsf{K}$ 
(ignoring the twist $\mathsf{J}_\sharp$).  
In these senses, we may say that the method in the present paper is 
a generalization of a variant of Kre\u{\i}n's theory in which the integral kernel is replaced by a different type. 
If we mention the differences, 
depending on the choice of integral kernels $k$ and $\Phi$, 
the operator $\mathsf{K}$ is isometric, but $f \mapsto \int \Phi''(x-y)f(y)\,dy$ is generally not. 
Also, some technical differences occur in the proof depending on 
whether the integral kernel is additive type $k(x+y)$ or difference type $\Phi''(x-y)$. 
On the other hand, if the generator $E$ of the de Branges space $\mathcal{H}(E)$ 
and the kernel $k=\mathsf{F}^{-1}\Theta$ of the operator $\mathsf{K}$ 
are not directly related as in the example of Section 3.3, 
the relationship with Kre\u{\i}n's theory becomes indirect. 

If we only aim to recover $H(t)$ from $\tau$ (or $\Theta$), 
which method is more useful will depend on the ease of 
handling with functions  $\Phi(t)$ and $k(x)$, 
but if we also take into account the recovery of the solution 
of the canonical system, 
there is more difference between the two methods. 
In Kre\u{\i}n's theory, the solution of 
\eqref{eq_20206_2} also generates the independent solutions 
$\phi(x,\lambda)$ and $\psi(x,\lambda)$ of the differential equation 
$(p y')' + \lambda py =0$ as 
\[
\aligned 
\phi(t;\lambda) 
&= \frac{1}{p(t)} \frac{d}{dt} \int_{0}^{t} q(s;t) \cos(s\sqrt{\lambda}) \, ds, \\
\psi(t;\lambda) 
&= \frac{1}{p(t)} \frac{d}{dt} \int_{0}^{t} q(s;t)\, \omega(s,\sqrt{\lambda}) \, ds, 
\endaligned 
\]
where $
\kappa \omega(t,\kappa) 
= \sin(\kappa t) +  \int_{0}^{t} H(t-s) \sin(\kappa s) \, ds$ 
and $\phi(0;\lambda)=1$, $\phi'(0;\lambda)=0$, $\psi(0;\lambda)=0$, $\psi'(0;\lambda)=1$. 
Using these solutions, 
the solution of the canonical system associated with $H(t)$ 
is obtained as 
\[
\begin{bmatrix} A(t,z) \\ B(t,z) \end{bmatrix}
=
\begin{bmatrix}
A(z)\psi'(t,z^2) - z^{-1}B(z)\phi'(t,z^2) \\  B(z)\phi(t,z^2) -zA(z)\psi(t,z^2)
\end{bmatrix}
\]
such that $E(z)=A(0,z)-iB(0,z)$ holds, where $A=(E+E^\sharp)/2$ and $B=i(E-E^\sharp)/2$. 
As this, the solution of the canonical system given by Theorems \ref{thm_2_1} and \ref{thm_2_6} 
can also be obtained from Kre\u{\i}n's theory, 
but the formula for the solution by \eqref{c_205} and \eqref{c_212} is 
somewhat direct and  simpler. 
This difference is the same compared to the theory in \cite{KrLa85} 
that deals with diagonal and non-diagonal $H$.

As the above, in the case of diagonal $H$, 
there are both similarities and differences between our method and Kre\u{\i}n's theory for strings.  
The advantages of our method are that 
it can be generalized to non-diagonal $H$ in a different way than \cite{KrLa85}, 
and that it can deal with chains of spaces 
that are not necessarily related to canonical systems. 
\subsection{Comparison with previous work.} 

To conclude this section, 
we comment on the difference of the operator $\mathsf{K}$ 
between this paper and \cite{Su18, Su19_1, Su19_3}. 
The first difference is that $\mathsf{K}$ is antilinear  in this paper, 
which is an essential ingredient that enables us the construction of non-diagonal $H$, 
whereas, in the latter, $\mathsf{K}$ was linear. 
Second, in this paper, $\mathsf{K}$ is defined as the composition of several unitary operators 
as in \eqref{c_202}, 
whereas in the latter, $\mathsf{K}$ was defined as the integral operator 
having the integral kernel defined by 
$K(x)=\frac{1}{2\pi}\int_{\Im(z)=c}u(z)e^{-izx} dz$ for large $c>0$. 
This second difference is reflected in the conditions assumed in the results, 
and each has advantages and disadvantages. 
As an example of what makes a significant difference, 
we take up the equality $\mathsf{K}^2=1$ and the support condition of $\mathsf{K}$, 
both are important in various discussions.
In the definition of this paper, 
$\mathsf{K}^2=1$ is obvious, 
but it is nontrivial that the support of $k(x)$ is contained in $[0,\infty)$, 
which corresponds to $u$ being an inner function. 
On the other hand, in \cite{Su18, Su19_1, Su19_3}, 
it is obvious from the settings that $K(x)$ is supported in $[0,\infty)$, 
but $\mathsf{K}^2=1$ is nontrivial and it relates whether $u$ is inner. 
Besides these, if $K(x)$ is discontinuous or distribution, 
the definition in this paper is more convenient. 
In any case, two different definitions of $\mathsf{K}$ can be related as in 
\cite[Theorem 5.1]{Su19_1}.

%
\section{Examples} \label{section_3}
%

We provide several concrete examples of unimodular functions 
satisfying (a part of) the conditions assumed in the results in Section \ref{section_2}. 
Those examples may help readers understand the meaning or necessity of conditions (O1)--(O8).

%
\subsection{Paley--Wiener spaces}
%

Let $E(z)=\exp(-iaz)$ with $a>0$ 
and put $u(z)=E^\sharp(z)/E(z)=\exp(2iaz)$. 
Then $u \in U_{\rm loc}^1(\R)$. 
We have $k(x)=\delta(x-2a)$, so $(\mathsf{K} f)(x) = \overline{f(2a-x)}$, 
and therefore $\mathcal{V}_t(u) = L^2(t,\infty) \cap L^2(-\infty,2a-t)$. 
Hence (O5) is satisfied, $t_0=a$,  
$\mathcal{V}_t(u) = L^2(t,2a-t)$ for $t<a$, 
and also (O8) is satisfied: $\mathcal{V}_{a}(u)=\{0\}$.  
We have $\mathsf{K}[t]=0$ for $t <a$ from the support condition of $k(x)$, 
so (O1),  (O4) and (O6) are satisfied, and $\Phi(t,x)=\Psi(t,x)=1$ for $x<t$ if $t<a$.
Therefore, \eqref{c_203} and \eqref{c_204} 
are solved as 
$\Phi(t,x) =1 - \mathsf{K}\mathsf{P}_t1(x)= 1 - \mathbf{1}_{[0,\infty)}(x-(2a-t))$ 
and 
$\Psi(t,x) =1 + \mathsf{K}\mathsf{P}_t1(x)= 1 + \mathbf{1}_{[0,\infty)}(x-(2a-t))$. 
Hence (O2), (O3), and (O7) are satisfied with $\Phi(t,t)=\Psi(t,t)=1$, and 
\[
\aligned 
A(t,z) = \cos((a-t)z), \quad 
-iB(t,z) = -i\sin((a-t)z). 
\endaligned 
\]
On the other hand, $Y_z^t(x) = \mathbf{1}_{(t,2a-t)}(x) e^{izx}$ and 
\[
J(t;z,w) 
= \frac{\overline{E(z)}E(w)}{2\pi}\langle Y_w^t,Y_z^t \rangle 
= \frac{\sin((a-t)(w-\bar{z}))}{\pi(w-\bar{z})}.
\]
This shows that equality \eqref{c_210} holds. 
See also \cite{Su18_2}, 
where the case that $E$ is an exponential polynomial with real coefficients 
is studied, 
and explicit formulas for $\phi^\pm(t,x)$, $A(t,z)$, and $B(t,z)$ 
are stated, although $\Phi(t,x)$ and $\Psi(t,x)$ are not specified.   

%
\subsection{One dimensional de Branges space} \label{section_3_2}
%

Let $E(z)=1-iz$ and put $u(z)=E^\sharp(z)/E(z)=(1+iz)/(1-iz)$. 
Then $u \in U_{\rm loc}^1(\R)$. 
We have $k(x)=-\delta(x)+2e^{-x}\mathbf{1}_{(0,\infty)}(x)$ 
and easily find that 
$\mathcal{V}_t(u) = \{0\}$ for $t>0$, 
$\mathcal{V}_0(u) = \C\, e^{-x} \mathbf{1}_{(0,\infty)}$, 
and 
$\mathcal{V}_t(u) = L^2(t,-t)+\C\, e^{-x} \mathbf{1}_{(-t,\infty)}$ for $t<0$. 
Hence (O5) is satisfied and $t_0=0$, 
but (O8) is {\it not} satisfied.  
We have $\mathsf{K}[t]=0$ for $t <0$, 
so (O1),  (O4), and (O6) are satisfied, and $\Phi(t,x)=\Psi(t,x)=1$ for $x<t$ if $t<0$.
Therefore, \eqref{c_203} and \eqref{c_204} 
are solved as 
$\Phi(t,x) =1 - \mathsf{K}\mathsf{P}_t1(x)= 1- \mathbf{1}_{(-t,\infty)}(x) (1-2 e^{-x-t} )$ 
and 
$\Psi(t,x) =1 + \mathsf{K}\mathsf{P}_t1(x)= 1+ \mathbf{1}_{(-t,\infty)}(x) (1-2 e^{-x-t} )$. 
Hence (O2), (O3), and (O7) are satisfied with $\Phi(t,t)=\Psi(t,t)=1$, and 
\[
A(t,z) 
= \cos(tz)+z\sin(tz), \quad 
-iB(t,z) 
= -i(z\cos(tz)-\sin(tz)).  
\]
On the other hand, we have 
\[
Y_z^t(x) = \frac{2i}{z+i}e^{-it(z-i)} \cdot e^{-x} \mathbf{1}_{(-t,\infty)}(x)
+ e^{izx}\mathbf{1}_{(t,-t)}(x), 
\] 
and 
\[
J(t;z,w) 
= \frac{\overline{E(z)}E(w)}{2\pi}\langle Y_w^t,Y_z^t \rangle 
= \frac{e^{-it(w-\bar{z})}}{\pi} 
- (w+i)(\bar{z}-i)\frac{\sin(t(w-\bar{z}))}{\pi(w-\bar{z})}
\]
for $t \leq 0$. We can check that equality \eqref{c_210} holds. 

For $0 < t <1$, we define $A(t,z)=1$ and $B(t,z)=z(1-t)$. Then 
\[ 
-\frac{d}{dt}
\begin{bmatrix}
A(t,z) \\ B(t,z)
\end{bmatrix}
= z 
\begin{bmatrix}
0 & -1 \\ 1 & 0
\end{bmatrix}
\begin{bmatrix}
1 & 0 \\ 0 & 0
\end{bmatrix}
\begin{bmatrix}
A(t,z) \\ B(t,z) 
\end{bmatrix}, \quad 0 < t < 1, \quad z \in \C,  
\]
$J(t;z,w)= (1/\pi)(1-t)$, and hence $J(t;z,w) \to 0$ as $t \to 1$. 
In other words, the structure Hamiltonian of $\mathcal{H}(E)$ 
is $H(t)={\rm diag}(1,0)$ on $t \in (0,1)$, 
but it cannot be obtained from the method in Section \ref{section_2}.

%
\subsection{De Branges spaces related to Hankel transform of order zero}
%

The example described here is based on the results of Burnol \cite{Burnol2011}. 
The section numbers in this part refers to that in \cite{Burnol2011}. 
Let $M(z)=\Gamma(\tfrac{1}{2}-iz)$ and put 
$u(z)=M^\sharp(z)/M(z)=\Gamma(\tfrac{1}{2}+iz)/\Gamma(\tfrac{1}{2}-iz)$. 
Then $u \in U_{\rm loc}^1(\R)$ and $u$ is meromorphic on $\C$ but not inner in $\C_+$. 
We have 
$k(x)=e^{x/2} J_0(2e^{x/2})$, and 
$u(z)= \int_{-\infty}^{\infty} k(x)  e^{izx} \, dx$ 
for $-1/4 < \Im(z) <1/2$, 
where 
$J_0(z)$ is the Bessel function of the first kind of order $0$ 
and the convergence of the Fourier integral is conditional if $\Im(z)\leq 1/4$. 
It is proved that  (O1) and (O5) are satisfied for all $t \in \R$ in \S5, 
and hence (O6) and (O8) are also satisfied. 
We put $a=e^t$, $b=e^x$, and 
\[
\phi^\pm(t,x) = \sqrt{ab} \left( 1\pm \frac{\partial}{\partial b} \right) I_0(2\sqrt{a(a-b)}),  
\]
where $I_0(w)$ is the modified Bessel function of the first kind of index zero. 
Then $\phi^\pm(t,x)$ are real-valued, 
and they satisfy \eqref{c_501} and \eqref{c_502} below by discussions in \S7.
Define
\[
\aligned 
\Phi(t,x) 
= 
1 - e^{-2e^t}
\int_{-\infty}^{x} \phi^-(t,y) dy, 
\quad 
\Psi(t,x) 
= 
1 + e^{2e^t}
\int_{-\infty}^{x} \phi^+(t,y) dy. 
\endaligned 
\]
Then they satisfy \eqref{c_203} and \eqref{c_204}, 
and it is easily find that (O2) is satisfied. 
Also, the equalities $\Phi(t,t) = e^{-2e^t}$ and $\Psi(t,t) = e^{2e^t}$ 
show that (O3) and (O6) are satisfied. 
The asymptotic formula $(\mathsf{K}\delta_t)(x)=k(x+t)\sim \exp(x/2)$ as $x \to -\infty$ 
shows that (O4) is satisfied. 
Further, we have 
\[
\aligned 
A(t,z) &= e^{2e^t+(t/2)}(K_{\frac{1}{2}-iz}(2e^t) + K_{\frac{1}{2}+iz}(2e^t)), \\
-iB(t,z) &= e^{-2e^t+(t/2)}(K_{\frac{1}{2}-iz}(2e^t) - K_{\frac{1}{2}+iz}(2e^t)), 
\endaligned 
\]
where $K_\nu(z)$ is the modified Bessel function of the second kind of index $\nu$. 
The explicit formula for  $Y_z^t$ cannot be found in \cite{Burnol2011}, 
and in fact, that is complicated to write down here,  
but the explicit formula for $\langle Y_w^t, Y_z^t \rangle$ 
can be found in \S7. For $z=0$, we have $Y_0^t(x) =\mathbf{1}_{(t,\infty)}(x)\Phi(t,x)$, 
which is implicitly dealt with in \S6.  
Anyway, we find the explicit formula for $J(t;z,w)$ 
via the second equality of \eqref{c_210}. 

%
\subsection{De Branges spaces arising from $L$-functions in the Selberg class}
%

The examples described here are based on the results of \cite{Su19_2}. 
Let $\mathcal{S}$ be the Selberg class of Dirichlet series $L(s)$. 
Typical examples of elements in $\mathcal{S}$ 
are number-theoretic zeta- and $L$-functions such as the Riemann zeta-function 
and Dirichlet $L$-functions. 
For every $L \in \mathcal{S}$, there exists 
a product $\gamma_L(s)$ of shifts of the $\Gamma$-function, 
a real number $Q_L$ and a complex number $c_L$ of the unit modulus 
such that $\xi_L(s):=c_L\,Q_L^{s}\gamma_L(s)L(s)$ satisfies 
the functional equation $\xi_L(s)=\xi_L^\sharp(1-s)$. We define 
\[
\Theta_L^{\omega,\nu}(z)
= \left(
\frac{\xi_L(1/2-\omega-iz)}{\xi_L(1/2+\omega-iz)}
\right)^\nu
\]
for $L \in \mathcal{S}$, $\omega \in \R_{>0}$, and $\nu \in \Z_{>0}$. 
Then $\Theta_L^{\omega,\nu}$ is a meromorphic inner function 
for all $\omega \geq 1/2$ and $\nu \in \Z_{>0}$ unconditionally 
and for all $0<\omega<1/2$  and $\nu \in \Z_{>0}$ 
if assuming the Grand Riemann Hypothesis (GRH) for $L$ 
(\cite[Proposition 2.2]{Su19_2}). 

We suppose that $\omega\nu d_L>1$, where $d_L$ is the degree of $L$, and 
that $\Theta_L^{\omega,\nu}$ is a meromorphic inner function.  
Then, by \cite[Theorem 5.1]{Su19_1} and \cite[Proposition 4.1]{Su19_2}, 
$\mathsf{K}$ defined by \eqref{c_202} for $u=\Theta_L^{\omega,\nu}$
equals to the integral operator having the continuous kernel 
$K(x+y)$ defined by $K(x)=\frac{1}{2\pi}\int_{\Im(z)=c} \Theta_L^{\omega,\nu}(z) e^{-izx}\,dz$ 
for $c>1/2+\omega$. 
Further,  (O2), (O3), (O4) are satisfied by 
\cite[Section 2.4]{Su19_1}, \cite[Proposition 4.1]{Su19_2}, and 
\cite[Proposition 2.3]{Su19_3}. 
And (O6) is satisfied by a similar argument as in the proof of 
\cite[Lemma 5.2, Proposition 5.1]{Su19_2}. 
Since Proposition \ref{prop_7_4} below can be applied, 
$\Phi(t,t)=\Psi(t,t)=1$ for $t<0$, 
and $\Phi(t,t)$, $\Psi(t,t)$ are continuous for $t \geq 0$. 
Therefore, (O7) is satisfied. 
Finally, since $t_0=\infty$ by 
\cite[Lemma 4.1]{Su19_1} and \cite[Proposition 4.1]{Su19_2}, 
(O8) is also satisfied. 
Hence, Theorem \ref{thm_2_6} can be applied to $u=\Theta_L^{\omega,\nu}$. 
Note that \cite{Su19_2} deals only with the case that 
$L(s)$ and $Q_L^{s}\gamma_L(s)L(s)$ take real-values on the real line, 
but essentially does not affect the discussions  
that prove the results referred to the above, 
and lead to the same results for general $L \in \mathcal{S}$. 

On the other hand, we may also generalize 
the unimodular function 
\[
u(z)=\exp\left[-2\eta \frac{\xi'}{\xi}\left(\frac{1}{2}-iz\right)\right]
\]
studied in \cite{Su18} to functions in the Selberg class, 
where $\xi$ is the Riemann xi-function and $\eta$ is a positive real number. 

%
\section{Proof of Theorem \ref{thm_2_1}} \label{section_4}
%

%
\subsection{Extension of $\mathsf{K}$ to $S'(\R)$} \label{section_4_1}
%

Let $(f,g) := \int f(x)g(x) \,dx$ be the pairing for $f \in S'(\R)$ and $g \in S(\R)$. 
Then the Fourier transform $\mathsf{F}$, 
the multiplication operator $\mathsf{M}_u$ for $u \in L^\infty(\R)$, 
and $\mathsf{J}^\sharp$ extend to $S'(\R)$ by 
$(\mathsf{F}f,g):=(f,\mathsf{F}g)$, 
$(\mathsf{M}_uf,g):=(f,\mathsf{M}_ug)$, and 
$(\mathsf{J}^\sharp f,g) = \overline{(f,\mathsf{J}^\sharp g)}$ 
for $f \in S'(\R)$ and $g \in S(\R)$, 
since $S(\R)$ is closed under these operations. 
Therefore, $\mathsf{K}$ extends to 
an antilinear involution on $S'(\R)$ by \eqref{c_202}. 

We also use the Hermitian pairing $\langle f, g \rangle:= (f,\bar{g}) = \overline{ (\bar{f}, g) }$ 
for $f \in S'(\R)$ and $g \in S(\R)$. 
Then, $\mathsf{K}$ is self-adjoint with respect to this Hermitian paring. 
In fact, we have $\overline{\mathsf{F}^{-1}h_1}
= (2\pi)^{-1}\mathsf{F}\mathsf{J}^\sharp h_1$, 
$\mathsf{J}^\sharp\mathsf{F}h_2= 2\pi \mathsf{F}^{-1} \overline{h_2}$, 
and therefore 
\[
\langle f, \mathsf{K}g \rangle 
= (f, \mathsf{F}\mathsf{J}^\sharp\mathsf{M}_u \mathsf{F}^{-1} \bar{g} )
= \overline{ (\mathsf{F}^{-1}\mathsf{M}_u\mathsf{J}^\sharp\mathsf{F}f,   \bar{g} ) } 
= \overline{ \langle \mathsf{K}f,   g \rangle }. 
\]

\begin{proposition} \label{prop_4_1}
Let $u \in U_{\rm loc}^1(\R)$.  
Then $\mathsf{K}\mathsf{P}_t 1 \in L_{\rm loc}^1(\R)$ and $\mathsf{K}[t]1 \in L^2(-\infty,t)$.  
\end{proposition}
\begin{proof} We calculate $\mathsf{K}\mathsf{P}_t 1$ as a tempered distribution. 
For $g \in S(\R)$, 
\[
\aligned 
\langle \mathsf{K}\mathsf{P}_t1,g \rangle 
& = \langle \mathsf{K}g, \mathsf{P}_t1 \rangle  
  = ( \mathsf{F}^{-1}\mathsf{M}_u\mathsf{J}^\sharp\mathsf{F}g, \mathsf{P}_t1 ) 
  = ( \mathsf{M}_u\mathsf{J}^\sharp\mathsf{F}g, \mathsf{F}^{-1}\mathsf{P}_t1 ) \\
& = ( (\mathsf{M}_u-u(0))\mathsf{J}^\sharp\mathsf{F}g, \mathsf{F}^{-1}\mathsf{P}_t1 ) 
+u(0) ( \mathsf{F}^{-1}\mathsf{J}^\sharp\mathsf{F}g, \mathsf{P}_t1 ). 
\endaligned 
\]
In the second term of the right-hand side, 
$\mathsf{F}^{-1}\mathsf{J}^\sharp\mathsf{F}g 
= \mathsf{F}^{-1}\mathsf{F}\mathsf{J}_\sharp g = \mathsf{J}_\sharp g$ 
by $\mathsf{J}^\sharp\mathsf{F}=\mathsf{F}\mathsf{J}_\sharp$. 
Thus 
\[
(\mathsf{F}^{-1}\mathsf{J}^\sharp\mathsf{F}g, \mathsf{P}_t1)
= (\mathsf{J}_\sharp g, \mathsf{P}_t1)
= \int \overline{g(-x)} \mathsf{P}_t1(x) \, dx
= \langle (1-\mathsf{P}_{-t})1,g \rangle. 
\]
To calculate the first term of the right-hand side, we recall 
the Fourier transform of unit step functions
\[
\mathsf{F}^{-1}\mathsf{P}_t1(z)
= \frac{e^{-itz}}{2} \left( \delta(z) - \frac{1}{\pi i}  \,{\rm p.v.}\frac{1}{z}\right), 
\]
where the distribution ${\rm p.v.}(1/x)$ is defined by 
\[
\left( {\rm p.v.}\frac{1}{x}, g(x)\right) = \lim_{\epsilon \to 0} \int_{|x|\geq \epsilon} \frac{g(x)}{x} \, dx
\] 
for $g \in S(\R)$. 
Therefore, 
\[
\aligned 
( (\mathsf{M}_u & -u(0)) \mathsf{J}^\sharp\mathsf{F}g, \mathsf{F}^{-1} \mathsf{P}_t1 ) \\
& =  \left( (u(z)-u(0)) \int \overline{g(-x)} e^{izx} \, dx, ~ 
-\frac{1}{2\pi i} e^{-itz} {\rm p.v.}\frac{1}{z} 
  \right) \\
&= \frac{1}{2\pi} 
\lim_{\epsilon \to 0}  
\int_{|z|>\epsilon}\left( \frac{u(z)-u(0)}{-iz} \int \overline{g(-x)} e^{izx} \, dx 
  \right) e^{-itz} \, dz \\
&=  
\int
\left( \frac{1}{2\pi} 
\int \left[ \frac{u(z)-u(0)}{-iz} e^{-itz} \right] e^{-izx} \, dz
  \right)  \overline{g(x)}  \, dx \\
& = \int h(x)\overline{g(x)} \, dx 
= \langle h, g \rangle, 
\endaligned 
\]
where 
\[
h(x):= \frac{1}{2\pi} 
\int \left[ \frac{u(z)-u(0)}{-iz} e^{-itz} \right] e^{-izx} \, dz \in L^2(\R)
\]
and the integral converges in the $L^2$ sense.  
(H\"{o}lder continuity of $u$ at $z=0$ is used here.)  
The above calculation is justified by the Cauchy--Schwarz inequality 
and Fubini's theorem.  Hence, 
\begin{equation} \label{c_401}
\mathsf{K}\mathsf{P}_t1
= u(0)(1-\mathsf{P}_{-t})1 + h, \quad h \in L^2(\R). 
\end{equation}
This shows the desired results, since $L^2(\R) \subset L_{\rm loc}^1(\R)$.  
\end{proof}

%
\subsection{Proof of (1), (2), (3), (4)}
%

Recall the solutions $\Phi = 1 - \mathsf{K}\varphi - \mathsf{K}\mathsf{P}_t1$ 
and $\Psi = 1 + \mathsf{K}\psi + \mathsf{K}\mathsf{P}_t1$ 
of \eqref{c_203} and \eqref{c_204} introduced in Section \ref{section_2_1}, 
where $\varphi=-(1+\mathsf{K}[t])^{-1}\mathsf{K}[t]1$ 
and  $\psi=(1-\mathsf{K}[t])^{-1}\mathsf{K}[t]1$. 
It is not hard to see $\varphi=\mathsf{P}_t\Phi-\mathsf{P}_t1$ 
and  $\psi=\mathsf{P}_t\Psi-\mathsf{P}_t1$. 
By Proposition \ref{prop_4_1}, 
$\Phi$ and $\Psi$ are tempered distributions at least. 
Therefore,  
$\mathsf{F}(1-\mathsf{P}_t)\Phi$ 
and  
$\mathsf{F}(1-\mathsf{P}_t)\Psi$ 
are always defined as tempered distributions, thus (1) is proved. 

To prove (2) and (3), we show that 
$\mathsf{F}(1-\mathsf{P}_t)\Phi$ 
and 
$\mathsf{F}(1-\mathsf{P}_t)\Psi$ 
are defined as holomophic functions on $\C_+$. 
First, $(\mathsf{F}(1-\mathsf{P}_t)1)(z)=e^{itz}/(-iz)$ for $\Im(z)>0$. 
Second, $\mathsf{F}(1-\mathsf{P}_t)\mathsf{K}\psi $ 
is defined for $\Im(z) > 0$, since $(1-\mathsf{P}_t)\mathsf{K}\psi \in L^2(t,\infty)$. 
Third, 
\[
(\mathsf{F}(1-\mathsf{P}_t)\mathsf{K}\mathsf{P}_t1)(z)
= u(0)\frac{e^{i|t|z}}{-iz} + (\mathsf{F}(1-\mathsf{P}_t)h)(z), \quad h \in L^2(\R)
\] 
for $\Im(z) > 0$ by \eqref{c_401}. 
Hence $\tilde{A}(t,z)$ and $\tilde{B}(t,z)$ are defined by \eqref{c_205} 
and holomorphic on $\C_+$. 
Moreover, $\lim_{z \to x}\tilde{A}(t,z)=\tilde{A}(t,x)$ 
and $\lim_{z \to x}\tilde{B}(t,z)=\tilde{B}(t,x)$ 
hold for almost all $x \in \R$, 
where $z$ tends to $x$ non-tangentially in $\C_+$, 
since $\lim_{z \to x} \mathsf{F}f(z) = \mathsf{F}f(x)$ 
for $f \in L^2(t,\infty)$ for almost all $x \in \R$. 

For (4), we extend $\tilde{A}(t,z)$ and $\tilde{B}(t,z)$ across the real line as follows.   
First, we observe that $\mathsf{K}1 = u(0)$ as a tempered distribution, 
because  
\[
\aligned 
\langle \mathsf{K}1, g \rangle
& = \langle  \mathsf{K}g, 1 \rangle 
 = \langle \mathsf{F}^{-1}\mathsf{F} \mathsf{K}g, 1 \rangle 
 =  (\mathsf{F}^{-1}\mathsf{F} \mathsf{K}g, 1 )
=  (\mathsf{F} \mathsf{K}g, \mathsf{F}^{-1}1 ) \\
& =  (\mathsf{F} \mathsf{K}g, \delta )
= \ \mathsf{F} \mathsf{K}g(0) 
= u(0)(\mathsf{F}g)^\sharp(0) 
= u(0)\overline{\mathsf{F}g(0)} 
= u(0) \langle 1, g \rangle
\endaligned 
\]
for $g \in S(\R)$. 
On the other hand, 
$(1-\mathsf{P}_t)\Psi 
= 1 - \mathsf{P}_t \Psi  + \mathsf{K}\mathsf{P}_t\Psi $ 
by \eqref{c_204}. 
Therefore,  
\[
\mathsf{K}(1-\mathsf{P}_t)\Psi 
= \mathsf{K}1 - \mathsf{K}\mathsf{P}_t \Psi  + \mathsf{P}_t\Psi 
= (u(0) + 1) -  (1-\mathsf{P}_t)\Psi .
\]
Using this and $(\mathsf{F}\frac{\partial}{\partial x} f)(z)=(-iz)(\mathsf{F}f)(z)$, 
we have 
\[
\aligned 
2 \tilde{A}(t,z) 
& = (-iz)(\mathsf{F}(1-\mathsf{P}_t)\Psi)(z) 
  = (\mathsf{F}\tfrac{\partial}{\partial x}(1-\mathsf{P}_t)\Psi)(z) \\
& = - (\mathsf{F}\tfrac{\partial}{\partial x} \mathsf{K}(1-\mathsf{P}_t)\Psi)(z) 
 = (\mathsf{F}\mathsf{K}\tfrac{\partial}{\partial x}(1-\mathsf{P}_t)\Psi)(z) \\
&= u(z) (J^\sharp \mathsf{F}\tfrac{\partial}{\partial x}(1-\mathsf{P}_t)\Psi)(z)
= 2\,u(z)\tilde{A}^\sharp(t,z)
\endaligned 
\]
for $z \in \R$. On the right-hand side, 
\[
\aligned 
2 \tilde{A}^\sharp(t,z)
& = (\mathsf{J}^\sharp\mathsf{F}\tfrac{\partial}{\partial x}(1-\mathsf{P}_t)\Psi)(z)
= (\mathsf{F}\tfrac{\partial}{\partial x}\mathsf{J}_\sharp(1-\mathsf{P}_t)\Psi)(z) \\
& = (\mathsf{F}\tfrac{\partial}{\partial x}\mathsf{P}_{-t}\Psi)(z) 
= (-iz)(\mathsf{F}\mathsf{P}_{-t}\mathsf{J}_\sharp\Psi)(z). 
\endaligned 
\]
Here 
$(\mathsf{F}\mathsf{P}_{-t}\mathsf{J}_\sharp \Psi)(z)$ 
on the right-hand side is defined for $\Im(z) \leq 0$ 
by $\Psi = 1 + \mathsf{K}\psi + \mathsf{K}\mathsf{P}_t1$ and  \eqref{c_401}. 
Hence $\tilde{A}(t,z)$ extends 
from $\C_+ \cup \R$ to $\C_+ \cup D$, and \eqref{c_206} holds in $D$. 
Analytic continuation and functional equation for $\tilde{B}(t,z)$ 
are proved in a similar argument. 
\medskip

If the (distribution) kernel $k=\mathsf{F}^{-1}u$ of $\mathsf{K}$ has support in $[0,\infty)$, 
we easily find that $\varphi$ and  $\psi$ have support in $[-t,t]$ or zero 
(cf. Section \ref{section_6_1}), 
and thus $\mathsf{F}\varphi$ and $\mathsf{F}\psi$ are entire functions, 
since a tempered distribution is a higher derivative of a continuous function. 
If the analytic continuations for $\mathsf{F}\varphi$ and $\mathsf{F}\psi$ beyond $\C_+ \cup \R$ 
are easily proved like these, 
the analytic continuations for $\tilde{A}(t,z)$ and $\tilde{B}(t,z)$ are also easily proved as follows.
We have 
$(1-\mathsf{P}_t)\Psi 
= (1-\mathsf{P}_t)1 - \psi
+ \mathsf{K}\mathsf{P}_t1
+ \mathsf{K}\psi$ 
with 
$\psi=(\mathsf{P}_t\Psi - \mathsf{P}_t 1)$ 
by \eqref{c_204}. 
Hence 
\[
(-iz)(\mathsf{F}(1-\mathsf{P}_t)\Psi )(z)
 = \Bigl[ e^{itz} + iz (\mathsf{F} \psi)(z) \Bigr]
 + u(z) \Bigl[
e^{itz}
+ iz
(\mathsf{F}\psi)(z)
\Bigr]^\sharp
\]
for $z \in \C_+ \cup \R$. This formula gives the analytic continuation of 
$\tilde{A}(t,z)$ according to the extended domain of $(\mathsf{F}\psi)(z)$ 
(and the domain of $u$). \hfill $\Box$

%
\subsection{Proof of (5)}
%

For $t<s$, we have 
\[
\Phi(t,x)-\Phi(s,x) = \mathsf{K}(\mathsf{P}_s-\mathsf{P}_t) 1 
+ \mathsf{K}\Bigl[ 
(1+\mathsf{K}[t])^{-1}\mathsf{K}[t]1 
-
(1+\mathsf{K}[s])^{-1}\mathsf{K}[s]1 
\Bigr]
\]
by $\Phi(t,x) = 1 - \mathsf{K}\mathsf{P}_t 1 + \mathsf{K}(1+\mathsf{K}[t])^{-1}\mathsf{K}[t]1$ 
(see the lines before \eqref{c_203}). 
The first term on the right-hand side tends to zero as $s \to t$ in $L^2(\R)$, 
since $\mathsf{K}$ is isometric. 
We find that the second term on the right-hand side 
also tends to zero as $s \to t$ in $L^2(\R)$ 
by the second resolvent equation 
\[
(1+\mathsf{K}[t])^{-1}\mathsf{K}[t]
 -
(1+\mathsf{K}[s])^{-1}\mathsf{K}[s] 
 = 
(1+\mathsf{K}[t])^{-1}
(\mathsf{K}[t]-\mathsf{K}[s])
(1+\mathsf{K}[s])^{-1}.
\]
Therefore, 
$\Vert \Phi(t,x) - \Phi(s,x) \Vert_{L^2(\R)} \to 0$ as $s \to t$. 
The same is true for $\Psi(t,x)$. Therefore, 
$\Vert z^{-1}(\tilde{A}(t,z) - \tilde{A}(s,z)) \Vert_{H^2(\C_+)} \to 0$ 
and 
$\Vert z^{-1}(\tilde{B}(t,z) - \tilde{B}(s,z)) \Vert_{H^2(\C_+)} \to 0$ 
as $s \to t$ by definition \eqref{c_205}.  
The latter implies that $\tilde{A}(t,z) - \tilde{A}(s,z) \to 0$ and $\tilde{B}(t,z) - \tilde{B}(s,z) \to 0$ 
as $s \to t$ pointwisely, since the norm convergence in the reproducing kernel Hilbert space 
$H^2(\C_+)$ implies the pointwise convergence. \hfill $\Box$

%
\subsection{Auxiliary results necessary to prove (6) and (7)}
%

\begin{lemma} \label{lem_4_2}
Let $\Phi$ and $\Psi$ be nonzero complex numbers. Then the pair of equations 
\[
\left\{
\aligned 
\Psi &= -i \beta \Psi  + \gamma \Phi, \\ 
\Phi &= \alpha \Psi  + i \beta \Phi
\endaligned 
\right.
\]
for real numbers $\alpha$, $\beta$, $\gamma$ has a unique solution 
\begin{equation} \label{c_402}
\aligned 
\alpha
& = \frac{|\Phi|^2}{\Re(\Phi\overline{\Psi})}
= \frac{1}{\Re(\Psi/\Phi)}, \quad 
\gamma
= \frac{|\Psi|^2}{\Re(\Phi\overline{\Psi})}
= \frac{1}{\Re(\Phi/\Psi)}, \\
& \qquad \beta
= \frac {\Im(\Phi\overline{\Psi})}{\Re(\Phi\overline{\Psi})}
= \frac {\Im(\Phi/\Psi)}{\Re(\Phi/\Psi)}
= -\frac {\Im(\Psi/\Phi)}{\Re(\Psi/\Phi)}
\endaligned 
\end{equation}
such that the symmetric matrix 
$\Bigl[
\begin{smallmatrix}
\alpha & \beta \\ \beta & \gamma 
\end{smallmatrix}
\Bigr]$ 
belongs to ${\rm SL}_2(\R)$. 
Moreover, 
$
\Bigl[
\begin{smallmatrix}
\alpha & \beta \\ \beta & \gamma 
\end{smallmatrix}
\Bigr]
$
is positive or negative definite according to the sign of $\Re(\Phi\overline{\Psi})$.  
\end{lemma}
\begin{proof} 
The given equation is equivalent to the linear equation
\[
\begin{bmatrix}
0 & \Im(\Psi) & \Re(\Phi) & - \Re(\Psi) \\
\Re(\Psi) & -\Im(\Phi) & 0 & - \Re(\Phi) \\
0 & -\Re(\Psi) & \Im(\Phi) & - \Im(\Psi) \\
\Im(\Psi) & \Re(\Phi) & 0 & - \Im(\Phi)
\end{bmatrix}
\begin{bmatrix}
\alpha \\ \beta \\ \gamma \\ 1
\end{bmatrix}
= 
\begin{bmatrix}
0 \\ 0 \\0 \\ 0
\end{bmatrix}. 
\]
The kernel of the matrix on the left-hand side is one-dimensional 
since $\Phi$ and $\Psi$ are non-zero. 
Hence the solution is unique. 
It can be confirmed by direct calculation that \eqref{c_402} solves the given equation   
and that $\alpha\gamma-\beta^2=1$. 
The eigenvalues of 
$
\Bigl[
\begin{smallmatrix}
\alpha & \beta \\ \beta & \gamma 
\end{smallmatrix}
\Bigr]
$
are 
\[
\frac{|\Phi|^2+|\Psi|^2 \pm 
\sqrt{(|\Phi|^2+|\Psi|^2)^2 -4(\Re(\Phi\overline{\Psi}))^2  }
}{2\,\Re(\Phi\overline{\Psi})}.
\]
These are nonzero real numbers for nonzero $\Phi$ and $\Psi$, 
because  
\[
\aligned 
(&|\Phi|^2+|\Psi|^2)^2 -4(\Re(\Phi\overline{\Psi}))^2 
=|\Phi-\Psi|^2|\Phi+\Psi|^2 \geq 0. 
\endaligned 
\]
The eigenvalues are both positive or both negative 
depending on the sign of $\Re(\Phi\overline{\Psi})$.  
\end{proof}

\begin{lemma} \label{lem_4_3} 
The following commutative relations hold in $S'(\R)$: 
\begin{equation} \label{c_403}
\frac{\partial}{\partial x}\mathsf{K} = - \mathsf{K}\frac{\partial}{\partial x}, 
\end{equation}
\begin{equation} \label{c_404}
\frac{\partial}{\partial x}\mathsf{P}_t f(x) = - \delta(x-t)f(x) + \mathsf{P}_t\frac{\partial}{\partial x}f(x).
\end{equation}
\end{lemma}
\begin{proof}
By the formula $\left(\mathsf{F}\frac{\partial}{\partial x}f\right)(z) = (-iz)(\mathsf{F}f)(z)$ 
for a tempered distribution $f$, 
\[
\mathsf{F}\frac{\partial}{\partial x}\mathsf{K}f(z)
= -iz (\mathsf{F}\mathsf{K}f)(z)
= - u(z)  \left[-iz(\mathsf{F}f) (z) \right]^\sharp 
= - \mathsf{F}\mathsf{K}\frac{\partial}{\partial x}f(z). 
\]
Hence \eqref{c_403} holds in $S'(\R)$. 
Because $\mathsf{P}_t f(x) = \mathbf{1}_{(-\infty,0)}(x-t)f(x)$, \eqref{c_404} is shown as 
\[
\aligned 
\frac{\partial}{\partial x}\mathsf{P}_t f(x)
& = \frac{\partial}{\partial x}(\mathbf{1}_{(-\infty,0)}(x-t) f(x)) \\
& = -\delta(x-t)f(x) + \mathbf{1}_{(-\infty,0)}(x-t)\frac{\partial}{\partial x}f(x)
\endaligned 
\]
by the product rule for derivatives. 
\end{proof}

\begin{lemma} \label{lem_4_4}
Condition (O2) implies
\begin{equation} \label{c_405}
\frac{\partial}{\partial t}\mathsf{P}_t X 
= \delta(x-t)X(t,x) + \mathsf{P}_t \frac{\partial}{\partial t}X, 
\end{equation}
\begin{equation} \label{c_406}
\frac{\partial}{\partial t}\mathsf{K} X
= \mathsf{K} \frac{\partial}{\partial t} X, 
\qquad 
\frac{\partial}{\partial t}\mathsf{K}\mathsf{P}_t X 
= \mathsf{K} \frac{\partial}{\partial t}\mathsf{P}_t X
\end{equation}
for $X \in \{\Psi(t,x), \Phi(t,x)\}$, 
where  $\Psi(t,x)$ and $\Phi(t,x)$ are functions in \eqref{c_203} and \eqref{c_204},  
and the projection $\mathsf{P}_t$ acts on $X=X(t,x)$ as a function of $x$. 
\end{lemma}
\begin{proof} 
Equation \eqref{c_405} is shown in a similar argument as \eqref{c_404}. 
The first equation of \eqref{c_406} is shown as follows by \eqref{c_202} and (O2): 
\[
\aligned
\mathsf{F}\mathsf{K} \frac{\partial}{\partial t} \Phi
& = \mathsf{M}_u \mathsf{F}\mathsf{J}_\sharp \frac{\partial}{\partial t} \Phi
= \mathsf{M}_u \frac{\partial}{\partial t} \mathsf{F}\mathsf{J}_\sharp \Phi
= \mathsf{F}\mathsf{F}^{-1}\frac{\partial}{\partial t} \mathsf{M}_u \mathsf{F}\mathsf{J}_\sharp \Phi \\
& = \mathsf{F}\frac{\partial}{\partial t}\mathsf{F}^{-1} \mathsf{M}_u \mathsf{F}\mathsf{J}_\sharp \Phi
= \mathsf{F}\frac{\partial}{\partial t}\mathsf{K} \Phi. 
\endaligned 
\]
The second equation of \eqref{c_406} is obtained from the first equation as follows. 
Applying $\partial/\partial t$ to \eqref{c_203}, 
we have 
$(\partial/\partial t)\Phi + (\partial/\partial t)\mathsf{K}\mathsf{P}_t \Phi= 0$. 
On the other hand, applying $\partial/\partial t$ 
to  \eqref{c_203} after acting $\mathsf{K}$, 
we have (i) 
$(\partial/\partial t)\mathsf{K}\Phi + (\partial/\partial t)\mathsf{P}_t \Phi= 0$. 
Thus 
$\mathsf{K}(\partial/\partial t)\Phi + (\partial/\partial t)\mathsf{P}_t \Phi= 0$
by the first equation of \eqref{c_406}, and therefore (ii)
$(\partial/\partial t)\Phi + \mathsf{K}(\partial/\partial t)\mathsf{P}_t \Phi= 0$. 
Comparing (i) and (ii), we obtain the second equation.  
The same is true for $\Psi$. 
\end{proof}

\begin{proposition} \label{prop_4_5} 
Assume that (O1), (O2), (O3), and (O4) are satisfied.  
Then the solutions $\Psi(t,x)$ and $\Phi(t,x)$ of  \eqref{c_203} and \eqref{c_204} 
satisfy the differential system 
\begin{equation} \label{c_407}
- 
\frac{\partial}{\partial t}
\begin{bmatrix}
\Psi(t,x) \\ i \Phi(t,x)
\end{bmatrix}
=
\begin{bmatrix}
0 & -1 \\ 1 & 0 
\end{bmatrix}
H(t)
\left( i\frac{\partial}{\partial x} \right) 
\begin{bmatrix}
\Psi(t,x) \\ i \Phi(t,x)
\end{bmatrix}, 
\end{equation}
where partial derivatives $\partial/\partial t$ and $\partial/\partial x$
are taken in the sense of distribution 
and $H(t)$ is defined by \eqref{c_208} and \eqref{c_209}. 
\end{proposition}
\begin{proof} 
First, we apply $(\partial/\partial t)$ to both sides of \eqref{c_204}. 
Then, 
\[
(\partial/\partial t)\Psi - \mathsf{K}(\partial/\partial t)\mathsf{P}_t \Psi = 0
\] 
by \eqref{c_406}, and further 
$(\partial/\partial t)\Psi 
-  \overline{\Psi(t,t)} \mathsf{K}\delta_t  - \mathsf{K}\mathsf{P}_t (\partial/\partial t)\Psi = 0$
by (O3) and \eqref{c_405}, where $\delta_t(x)=\delta(x-t)$. Hence, 
\begin{equation} \label{c_408}
\frac{\partial}{\partial t}\Psi - \mathsf{K}\mathsf{P}_t \frac{\partial}{\partial t}\Psi
= \overline{\Psi(t,t)} \mathsf{K}\delta_t.  
\end{equation}
Second, we apply $(\partial/\partial x)$ to both sides of \eqref{c_204}. 
Then, 
\[(\partial/\partial x)\Psi + \mathsf{K}(\partial/\partial x)\mathsf{P}_t \Psi = 0\] 
by \eqref{c_403}, and further 
$(\partial/\partial x)\Psi 
-  \overline{\Psi(t,t)} \mathsf{K}\delta_t  + \mathsf{K}\mathsf{P}_t (\partial/\partial x)\Psi = 0$
by (O3) and \eqref{c_404}. Hence, 
\begin{equation} \label{c_409}
\frac{\partial}{\partial x}\Psi + \mathsf{K}\mathsf{P}_t \frac{\partial}{\partial x}\Psi
= 
\overline{\Psi(t,t)} \mathsf{K}\delta_t, 
\end{equation}
and therefore, 
\begin{equation} \label{c_410}
i \beta(t)\frac{\partial}{\partial x}\Psi 
  - \mathsf{K}\mathsf{P}_t (i \beta(t))\frac{\partial}{\partial x}\Psi
= \overline{(-i \beta(t))\Psi(t,t)} \mathsf{K}\delta_t 
\end{equation}
for any real number $\beta(t)$. 
Third,  we apply $(\partial/\partial x)$ to both sides of \eqref{c_203}. 
Then, 
$(\partial/\partial x)\Phi - \mathsf{K}(\partial/\partial x)\mathsf{P}_t \Phi = 0$ 
by \eqref{c_403}, and further 
$(\partial/\partial x)\Phi 
+  \overline{\Phi(t,t)} \mathsf{K}\delta_t  - \mathsf{K}\mathsf{P}_t (\partial/\partial x)\Phi = 0$
by (O3) and \eqref{c_404}. Hence, 
\begin{equation} \label{c_411}
\quad \frac{\partial}{\partial x}\Phi
- \mathsf{K}\mathsf{P}_t \frac{\partial}{\partial x}\Phi
= -  \overline{\Phi(t,t)} \mathsf{K}\delta_t, 
\end{equation}
and therefore, 
\begin{equation} \label{c_412}
-\gamma(t)
\frac{\partial}{\partial x}\Phi
 + \mathsf{K}\mathsf{P}_t \gamma(t)\frac{\partial}{\partial x}\Phi
= \overline{\gamma(t)\Phi(t,t)} \mathsf{K}\delta_t
\end{equation}
for any real number $\gamma(t)$. 
Adding \eqref{c_410} and \eqref{c_412}, 
\[
\left( 1 - \mathsf{K}\mathsf{P}_t \right) \left[i \beta(t)\frac{\partial}{\partial x} \Psi
-\gamma(t)\frac{\partial}{\partial x} \Phi \right] 
= \overline{((-i \beta(t))\Psi(t,t)+\gamma(t)\Phi(t,t))} \mathsf{K}\delta_t. 
\]
The right-hand side is equal to $\overline{\Psi(t,t)}\mathsf{K}\delta_t$ 
by Lemma \ref{lem_4_2} if we take $\beta(t)$ and $\gamma(t)$ as in \eqref{c_209}.  
Comparing the obtained equality with \eqref{c_408}, we have
\[
\aligned 
\frac{\partial}{\partial t} \Psi(t,x)
& =  i \beta(t)\frac{\partial}{\partial x} \Psi(t,x) -  \gamma(t)\frac{\partial}{\partial x} \Phi(t,x)
\endaligned 
\]
by (O4). Hence the first line of \eqref{c_407} is obtained. A similar argument gives  
\begin{equation} \label{c_413}
\frac{\partial}{\partial t} \Phi
+ \mathsf{K}\mathsf{P}_t \frac{\partial}{\partial t}\Phi 
= -\overline{\Phi(t,t)}\mathsf{K}\delta_t. 
\end{equation}
This and  $\Phi(t,t) = \alpha(t)\Psi(t,t)  + i \beta(t)\Phi(t,t)$ 
with \eqref{c_209} 
lead to
\[
\aligned 
 \frac{\partial}{\partial t}\Phi(t,x)
&=
- \alpha(t)\frac{\partial}{\partial x} \Psi(t,x) - i \beta(t)\frac{\partial}{\partial x}\Phi(t,x)  
\endaligned 
\]
by Lemma \ref{lem_4_2}. 
Hence the second line of \eqref{c_407} is obtained.  
\end{proof}

%
\subsection{Proof of (6) and (7)}
%

For (6), we obtain \eqref{c_207} 
by applying the projection $(1-\mathsf{P}_t)$ 
to both sides of \eqref{c_407}, then taking their Fourier transform, 
and finally extend them to $\C_+ \cup D$, 
since $\mathsf{F}$ and $\partial/\partial t$ are commutative by (O2). 

The first half of (7) follows from Lemma \ref{lem_4_2}. 
If $u$ is symmetric, $k=\mathsf{F}^{-1}u$ is real-valued, 
since $\bar{k}=\mathsf{F}^{-1}(u^\sharp(-z))$. 
Therefore, if $\Phi$ solves \eqref{c_203}, 
then $\overline{\Phi}$ also solves it. 
Hence $\Phi=\overline{\Phi}$ by (O4) 
and thus $H(t)$ is diagonal by definition \eqref{c_209}. \hfill $\Box$

%
\section{Proof of results in Section \ref{section_2_2}} \label{section_5}
%

%
\subsection{Auxiliary results needed for the proof.}
%

%
To calculate the reproducing kernel of $\mathsf{F}(\mathcal{V}_t(u))$ explicitly, 
we study the solutions of the equations
\begin{equation} \label{c_501}
\phi^+ + \mathsf{K}\mathsf{P}_t\phi^+= \mathsf{K}\delta_t, 
\end{equation}
\begin{equation} \label{c_502}
\phi^- - \mathsf{K}\mathsf{P}_t\phi^-= \mathsf{K}\delta_t
\end{equation}
for $\phi^\pm$ in the space of tempered distributions $S'(\R)$, 
where $\delta_t(x)=\delta(x-t)$ as before. 

\begin{proposition} \label{prop_5_1} 
Suppose that $\Vert \mathsf{K}[t] \Vert_{\rm op}<1$ 
(as an operator on $L^2(-\infty,t)$) and that (O3), (O4) are satisfied. 
Then equations \eqref{c_501} and \eqref{c_502} 
have unique solutions in $S'(\R)$. 
\end{proposition}
\begin{proof} 
From the shape of equations, 
solutions of \eqref{c_501} and \eqref{c_502} 
are uniquely determined by their projection to $(-\infty,t)$. 
Hence if $\mathsf{P}_t\mathsf{K}\delta_t \in L^2(-\infty,t)$, 
$\Vert \mathsf{K}[t] \Vert_{\rm op}<1$ guarantees the existence and uniqueness of the solutions 
under the condition 
$\mathsf{P}_t\phi^\pm \in L^2(-\infty,t)$. 
If $\mathsf{P}_t\mathsf{K}\delta_t$ does not belong to $L^2(-\infty,t)$, 
the latter half of (O4) guarantees the uniqueness of the solutions in $S'(\R)$. 
The existence of the solutions is shown by constructing them concretely. 
From \eqref{c_409} and \eqref{c_411}, 
we find that $\phi^\pm$ defined by 
\begin{equation} \label{c_503}
\aligned 
\phi^+=\phi^+(t,x)
& =\frac{\Re(\Phi(t,t))}{\Re(\overline{\Psi(t,t)}\Phi(t,t))} 
\frac{\partial}{\partial x}\Psi(t,x) \\
& \qquad - i \, 
\frac{\Im(\Psi(t,t))}{\Re(\overline{\Psi(t,t)}\Phi(t,t))}
\frac{\partial}{\partial x}\Phi(t,x), 
\endaligned 
\end{equation}
\begin{equation} \label{c_504}
\aligned 
\phi^-=\phi^-(t,x)
& =i\,\frac{\Im(\Phi(t,t))}{\Re(\overline{\Psi(t,t)}\Phi(t,t))}
\frac{\partial}{\partial x}\Psi(t,x) \\
& \qquad -\frac{\Re(\Psi(t,t))}{\Re(\overline{\Psi(t,t)}\Phi(t,t))}
\frac{\partial}{\partial x}\Phi(t,x)
\endaligned
\end{equation}
solve the equations \eqref{c_501} and \eqref{c_502}, respectively. 
\end{proof}

If  (O2) is added to the assumptions of Proposition \ref{prop_5_1}, 
we  find that $\phi^\pm$ defined by 
\begin{equation} \label{c_505}
\aligned 
\phi^+=\phi^+(t,x) 
&= - \frac{\Re(\Psi(t,t))}{\Re(\overline{\Phi(t,t)}\Psi(t,t))}
\frac{\partial}{\partial t}\Phi(t,x) \\
& \qquad + i
\frac{\Im(\Phi(t,t))}{\Re(\overline{\Phi(t,t)}\Psi(t,t))}
\frac{\partial}{\partial t}\Psi(t,x), 
\endaligned 
\end{equation}
\begin{equation} \label{c_506}
\aligned 
\phi^-=\phi^-(t,x)
&=-i\frac{\Im(\Psi(t,t))}{\Re(\overline{\Phi(t,t)}\Psi(t,t))}
\frac{\partial}{\partial t}\Phi(t,x) \\
& \qquad +
\frac{\Re(\Phi(t,t)) }{\Re(\overline{\Phi(t,t)}\Psi(t,t))}
\frac{\partial}{\partial t}\Psi(t,x)
\endaligned
\end{equation}
also solve equations \eqref{c_501} and \eqref{c_502}, respectively, 
by \eqref{c_408} and \eqref{c_413}. 
Therefore, by comparing the right-hand sides of 
\eqref{c_503} and \eqref{c_505}, 
and
\eqref{c_504} and \eqref{c_506}, 
we get the system \eqref{c_407} again. 
Also, the following holds immediately from Proposition \ref{prop_5_1}. 

\begin{proposition} \label{prop_5_2} 
Suppose that $\Vert \mathsf{K}[t] \Vert_{\rm op}<1$ 
and that (O3), (O4) are satisfied. Let $a_0$, $a_1$, $b_0$, $b_1$ be real numbers. 
Then, 
$f + \mathsf{K}\mathsf{P}_t f = (a_0+ia_1) \mathsf{K}\delta_t$ 
for some $f \in S'(\R)$ implies $f= a_0 \phi_t^++i a_1\phi_t^-$, 
and  
$g - \mathsf{K}\mathsf{P}_t g = (b_0+ib_1) \mathsf{K}\delta_t$ 
for some $g \in S'(\R)$ implies $g=b_0\phi_t^-+ib_1\phi_t^+$. 
Conversely, 
$f + \mathsf{K}\mathsf{P}_t f = (a_0+ia_1) \mathsf{K}\delta_t$ 
and 
$g - \mathsf{K}\mathsf{P}_t g = (b_0+ib_1) \mathsf{K}\delta_t$ 
imply
\[
\phi^+(t,x)=\frac{1}{\Re((a_0+ia_1)(b_0-ib_1))}(b_0f(t,x)-ia_1g(t,x)), 
\]
\[
\phi^-(t,x)=\frac{1}{\Re((a_0+ia_1)(b_0-ib_1))}(-ib_1f(t,x)+a_0 g(t,x)). 
\]
\end{proposition}

For $t \in \R$, we identify $\mathsf{F}L^2(t,\infty)$ 
with $e^{itz}H^2(\C_+)$ as usual by the Poisson integral formula. 
Then, 
$\mathsf{F}L^2(t,\infty)=e^{itz}H^2(\C_+)$ 
is a reproducing kernel Hilbert space consisting of 
holomorphic functions on $\C_+$. 
In particular, the evaluation $F \mapsto F(z)$ is continuous for all $z \in \C_+$. 
The reproducing kernel is 
$\frac{1}{2\pi}
\langle (1-\mathsf{P}_t)e_w,(1-\mathsf{P}_t)e_z \rangle
= ie^{it(w-\bar{z})}/(2\pi(w-\bar{z}))$, 
where $e_z(x)=\exp(izx)$. 

\begin{lemma} \label{lem_5_3} 
Let $t \in \R$. 
If $\mathcal{V}_t(u) \not=\{0\}$, then 
there exists $Y_z^t \in \mathcal{V}_t(u)$ for each $z \in \C_+$ 
such that 
$\langle f, \overline{Y_z^t} \rangle=(\mathsf{F}f)(z)$ 
holds for all $f \in \mathcal{V}_t(u)$. 
Actually, $Y_z^t$ is the orthogonal projection of 
$(1-\mathsf{P}_t)e_z \in L^2(\R)$ to $\mathcal{V}_t(u)$. 
\end{lemma}
\begin{proof} 
$\mathsf{F}(\mathcal{V}_t(u))$ is a reproducing kernel Hilbert space 
consisting of functions on $\C_+$, 
since it is a closed subspace of 
$\mathsf{F}L^2(t,\infty)$ by definition. 
Therefore, $\mathsf{F}f \mapsto \mathsf{F}f(z)$ is continuous 
on $\mathsf{F}(\mathcal{V}_t(u))$ for $z \in \C_+$. 
Hence $f \mapsto \mathsf{F}f(z)$ is a linear continuous functional on $\mathcal{V}_t(u)$, 
and thus $Y_z^t$ exists by the Riesz representation theorem. 
If we have the decomposition 
$(1-\mathsf{P}_t)e_z= P_z^t+Q_z^t$ with  
$P_z^t \in \mathcal{V}_t(u)$ and $Q_z^t \in \mathcal{V}_t(u)^\perp$,  
\[
\aligned 
\int f(x)e^{izx} \, dx
& = \int f(x)P_z^t(x) \, dx + \int f(x)Q_z^t(x) \, dx \\
& = \int f(x)P_z^t(x) \, dx + 0 = \langle f, \overline{P_z^t}\rangle
\endaligned 
\]  
for all $f \in \mathcal{V}_t(u)$. 
Thus $Y_z^t$ coincides with the orthogonal projection $P_z^t$ of $(1-\mathsf{P}_t)e_z$.  
\end{proof}

\begin{lemma} \label{lem_5_4} 
Let $u \in U_{\rm loc}^1(\R)$ with $D \cap \C_-\not=\emptyset$ 
and let $t \in \R$. 
If $\mathcal{V}_t(u) \not=\{0\}$, 
each function $F \in \mathsf{F}(\mathcal{V}_t(u))$ 
extends to a function on $\C_+ \cup D$ 
and is meromorphic on $D \cap \C_-$. 
If $u$ is holomorphic in a neighborhood of the interval $(a,b) \subset \R$, 
all $F \in \mathsf{F}(\mathcal{V}_t(u))$ are holomorphic there. 
Moreover, $\langle f, \overline{Y_z^t} \rangle=(\mathsf{F}f)(z)$ 
holds for all $f \in \mathcal{V}_t(u)$ and $z \in \C_+ \cup D$. 
\end{lemma}
\begin{proof} Let $F=\mathsf{F}f$ for $f\in\mathcal{V}_t(u)$. Then 
$F=\mathsf{F}\mathsf{K}(\mathsf{K}f)
=\mathsf{M}_u\mathsf{J}^\sharp\mathsf{F}\mathsf{K}f
=\mathsf{M}_u\mathsf{F}\mathsf{J}_\sharp\mathsf{K}f$. 
On the right-hand side, $\mathsf{F}\mathsf{J}_\sharp\mathsf{K}f$ 
is defined and $F \mapsto F(z)$ is continuous on $\mathsf{F(}\mathcal{V}_t(u))$ for $z \in \C_- \cup \R$, 
since $\mathsf{J}_\sharp\mathsf{K}f \in L^2(-\infty,-t)$. 
We have $\lim_{z \to x}u(z)=u(x)$ and 
$\lim_{z \to x}(\mathsf{F}f)(z)=(\mathsf{F}f)(x)$ 
for almost all $x \in \R$ 
if $z$ tends to $x$ non-tangentially inside $\C_+$ and $\C_-$. 
Hence $F$ is holomorphic in a neighborhood of $(a,b) \subset \R$ 
if $u$ is holomorphic there. 
The evaluation $f \mapsto \mathsf{F}f(z)$ 
is continuous for $z \in \C_+$ and $z \in \C_-\cap D$, 
and therefore it is also continuous for almost all $z \in \R$ 
by the Banach--Steinhaus theorem. 
Hence there exists $Y_z^t \in \mathcal{V}_t(u)$ 
such that $\langle f, \overline{Y_z^t} \rangle=(\mathsf{F}f)(z)$ 
for $z \in \C_-$ and for almost all $z \in \R$. 
\end{proof}

\begin{lemma} \label{lem_5_5}
Let $t \in \R$. 
Suppose that $\Vert \mathsf{K}[t] \Vert_{\rm op} < 1$. 
Then, 
\begin{equation} \label{c_507}
\mathcal{V}_t(u)^{\perp} = L^2(-\infty,t) + \mathsf{K}(L^2(-\infty,t)). 
\end{equation}
\end{lemma}
\begin{proof} 
It is proved by almost the same argument as the proof of \cite[Lemma 4.2]{Su19_1}. 
\end{proof}

\begin{proposition} \label{prop_5_6}
Let $t \in \R$ and let $e_z(x)=\exp(izx)$ for  $z \in \C_+$. 
Suppose that 
$\Vert \mathsf{K}[t] \Vert_{\rm op} < 1$ and $\mathcal{V}_t(u) \not=\{0\}$ . 
Then the equations 
\begin{equation} \label{c_508}
(a_z^t-\mathsf{P}_te_z) + \mathsf{K}\mathsf{P}_t(a_z^t-\mathsf{P}_te_z) 
= (1-\mathsf{P}_t)e_z + \mathsf{K}(1-\mathsf{P}_t)e_z, 
\end{equation}
\begin{equation} \label{c_509}
(b_z^t-\mathsf{P}_te_z) - \mathsf{K}\mathsf{P}_t (b_z^t-\mathsf{P}_te_z) 
= (1-\mathsf{P}_t)e_z - \mathsf{K}(1-\mathsf{P}_t)e_z
\end{equation}
for functions $a_z^t=a_z^t(x)$ and $b_z^t=b_z^t(x)$ on $\R$ 
have unique solutions with conditions  
$a_z^t-\mathsf{P}_te_z \in L^2(\R)$ 
and $b_z^t-\mathsf{P}_te_z \in L^2(\R)$. Moreover, 
\begin{equation} \label{c_510}
Y_z^t = (1-\mathsf{P}_t)\frac{1}{2}(a_z^t+b_z^t). 
\end{equation}
\end{proposition}
\begin{remark} 
If $\mathsf{K}e_z$ makes sense, 
equations \eqref{c_508} and \eqref{c_509} are simplified as 
\begin{equation*} 
a_z^t+ \mathsf{K}\mathsf{P}_ta_z^t
= e_z + \mathsf{K}e_z, 
\qquad 
b_z^t- \mathsf{K}\mathsf{P}_t b_z^t
= e_z - \mathsf{K}e_z, 
\end{equation*}
and their solutions 
are considered according to the class of $\mathsf{K}e_z$.
\end{remark}
\begin{proof} Since $(1-\mathsf{P}_t)e_z \in L^2(\R)$ for $z \in \C_+$, 
$\mathsf{K}(1-\mathsf{P}_t)e_z$ belongs to $L^2(\R)$. 
Therefore, \eqref{c_508} and \eqref{c_509} 
are equations for functions 
$(a_z^t-\mathsf{P}_te_z)$ and $(b_z^t-\mathsf{P}_te_z)$ in $L^2(\R)$.  
Multiplying by $\mathsf{P}_t$ on both sides 
and then substituting the obtained formulas for $\mathsf{P}_t(a_z^t-\mathsf{P}_te_z)$ 
and $\mathsf{P}_t(b_z^t-\mathsf{P}_te_z)$ into \eqref{c_508} and \eqref{c_509}, 
we find that   
\begin{equation} \label{c_511}
\aligned 
a_z^t &= e_z +\mathsf{K}(1-\mathsf{P}_t)e_z - \mathsf{K}(1+\mathsf{K}[t])^{-1}\mathsf{P}_t\mathsf{K}(1-\mathsf{P}_t)e_z, \\
b_z^t &= e_z -\mathsf{K}(1-\mathsf{P}_t)e_z - \mathsf{K}(1-\mathsf{K}[t])^{-1}\mathsf{P}_t\mathsf{K}(1-\mathsf{P}_t)e_z
\endaligned 
\end{equation}
are unique solutions of \eqref{c_508} and \eqref{c_509} 
with conditions $a_z^t-\mathsf{P}_te_z, \, b_z^t-\mathsf{P}_te_z \in L^2(\R)$. 
Formulas in \eqref{c_511} show that 
both $(1-\mathsf{P}_t)a_z^t$ and $(1-\mathsf{P}_t)b_z^t$ belong to $L^2(t,\infty)$.  

Let us prove the formula \eqref{c_510}. 
By \eqref{c_507}, there exists unique vectors $u_z^t$ and $v_z^t$ in $L^2(-\infty,t)$ such that 
\begin{equation} \label{c_512}
(1-\mathsf{P}_t)e_z=Y_z^t + u_z^t+\mathsf{K}v_z^t. 
\end{equation}
Put $U_z^t=(a_z^t+b_z^t)/2 -\mathsf{P}_te_z$ and $V_z^t=(a_z^t-b_z^t)/2$. 
Then, $U_z^t + \mathsf{K}\mathsf{P}_tV_z^t=(1-\mathsf{P}_t)e_z$ and 
$V_z^t + \mathsf{K}\mathsf{P}_tU_z^t=\mathsf{K}(1-\mathsf{P}_t)e_z$ 
by \eqref{c_508} and \eqref{c_509}. By multiplying $\mathsf{K}$ on both sides of the second equation,
$(1-\mathsf{P}_t)e_z = \mathsf{K}(1-\mathsf{P}_t)V_z^t + \mathsf{P}_tU_z^t + \mathsf{K}\mathsf{P}_tV_z^t$. 
Therefore, $Y_z^t=\mathsf{K}(1-\mathsf{P}_t)V_z^t$, 
$u_z^t=\mathsf{P}_tU_z^t$ and $v_z^t=\mathsf{P}_tV_z^t$. 
Moreover, 
\[
\aligned 
Y_z^t
& =\mathsf{K}(1-\mathsf{P}_t)V_z^t
=\mathsf{K}V_z^t-\mathsf{K}\mathsf{P}_tV_z^t \\
& =((1-\mathsf{P}_t)e_z -\mathsf{P}_tU_z^t) + (U_z^t - (1-\mathsf{P}_t)e_z) = (1-\mathsf{P}_t)U_z^t . 
\endaligned 
\]
Hence, $Y_z^t$ belongs to $\mathcal{V}_t(u)$ 
and formula \eqref{c_510} follows from the uniqueness of decomposition \eqref{c_512}.
\end{proof}

\begin{proposition} \label{prop_5_8}
Let $t \in \R$ and $z \in \C_+$. 
Suppose that $\Vert \mathsf{K}[t] \Vert_{\rm op}<1$ 
and that (O3), (O4) are satisfied. 
Then the following equality holds: 
\begin{equation} \label{c_513}
\aligned
\left( \frac{\partial}{\partial x}  -iz \right) & 
(a_z^t(x)+ b_z^t(x)) \\
& = \frac{1}{2} \Bigl(
(a_z^t(t)+ b_z^t(t))-\overline{(a_z^t(t)-b_z^t(t))}
\Bigr) \phi^+(t,x) \\
& \quad  
- \frac{1}{2} \Bigl(
(a_z^t(t)+ b_z^t(t))+\overline{(a_z^t(t)-b_z^t(t))}
\Bigr) \phi^-(t,x). 
\endaligned 
\end{equation}
\end{proposition}  
\begin{proof} 
Equation \eqref{c_513} is proved by showing that both sides satisfy the same equation 
having a unique solution. 
By taking the sum and difference of \eqref{c_508} and \eqref{c_509}, 
we obtain
\begin{equation} \label{c_514} 
(a_z^t + b_z^t-2\mathsf{P}_te_z) 
+ \mathsf{K}\mathsf{P}_t (a_z^t - b_z^t) 
= 2(1-\mathsf{P}_t)e_z, 
\end{equation}
\begin{equation} \label{c_515} 
(a_z^t - b_z^t) + \mathsf{K}\mathsf{P}_t (a_z^t + b_z^t - 2\mathsf{P}_te_z) 
= 2 \, \mathsf{K}(1-\mathsf{P}_t)e_z, 
\end{equation}
respectively. 
Substituting $(a_z^t + b_z^t - 2\mathsf{P}_te_z)$ 
with $(a_z^t + b_z^t -2 e_z)$ 
in \eqref{c_515}, since it doesn't change the equation, 
then multiplying by $\mathsf{K}\mathsf{P}_t$ and subtracting from \eqref{c_514} yields
\begin{equation} \label{c_516}
(a_z^t + b_z^t) 
- \mathsf{K}\mathsf{P}_t \mathsf{K}\mathsf{P}_t (a_z^t + b_z^t - 2e_z) 
= 2e_z - 2 \mathsf{K}\mathsf{P}_t \mathsf{K}(1-\mathsf{P}_t)e_z.
\end{equation}
Using \eqref{c_403} and \eqref{c_404}, repeatedly, 
the derivative of the second term of the left-hand side 
and the derivative of  the right-hand side 
with respect to $x$ are calculated as 
\[
\aligned 
\frac{\partial}{\partial x} &
\mathsf{K}\mathsf{P}_t \mathsf{K}\mathsf{P}_t (a_z^t + b_z^t - 2e_z) \\
&= \mathsf{K}\delta_t 
\mathsf{P}_t \mathsf{K}\mathsf{P}_t (a_z^t + b_z^t - 2e_z) 
-\mathsf{K}\mathsf{P}_t \mathsf{K}
\delta_t (a_z^t + b_z^t - 2e_z) \\
& \quad +\mathsf{K}\mathsf{P}_t \mathsf{K}
\mathsf{P}_t \frac{\partial}{\partial x} (a_z^t + b_z^t - 2e_z)
\endaligned 
\]
and
\[
\aligned 
\frac{\partial}{\partial x}
\Bigl(
2e_z
-
2\mathsf{K}\mathsf{P}_t \mathsf{K}(1-\mathsf{P}_t)e_z
\Bigr)
& =  2iz ( e_z - \mathsf{K}\mathsf{P}_t \mathsf{K}(1-\mathsf{P}_t)e_z ) \\
& \quad - 2\mathsf{K}\mathsf{P}_t \mathsf{K}\delta_t e_z 
- 2\mathsf{K}(\delta_t\mathsf{K}(1-\mathsf{P}_t)e_z), 
\endaligned 
\]
respectively. Therefore, 
\[
\aligned 
\frac{\partial}{\partial x}(a_z^t + b_z^t) 
& - \mathsf{K}\mathsf{P}_t \mathsf{K}\mathsf{P}_t \frac{\partial}{\partial x}(a_z^t + b_z^t - 2e_z) \\
& = 2iz \Bigl( e_z - \mathsf{K}\mathsf{P}_t \mathsf{K}(1-\mathsf{P}_t)e_z \Bigr) 
- 2\mathsf{K}\mathsf{P}_t \mathsf{K}\delta_t e_z  
- 2\mathsf{K}(\delta_t\mathsf{K}(1-\mathsf{P}_t)e_z) \\
& \quad + \mathsf{K}\delta_t 
\mathsf{P}_t \mathsf{K}\mathsf{P}_t (a_z^t + b_z^t - 2e_z) 
- \mathsf{K}\mathsf{P}_t \mathsf{K}
\delta_t (a_z^t + b_z^t - 2e_z). 
\endaligned 
\]
Multiplying $(-iz)$ by \eqref{c_516} and then adding to this, 
\[
\aligned 
\left( \frac{\partial}{\partial x}-iz \right) & (a_z^t + b_z^t)
- \mathsf{K}\mathsf{P}_t \mathsf{K}\mathsf{P}_t 
\left( \frac{\partial}{\partial x}-iz \right) (a_z^t + b_z^t-2e_z) \\
& =  
- 2\mathsf{K}\mathsf{P}_t \mathsf{K}\delta_t e_z  
- 2\mathsf{K}(\delta_t\mathsf{K}(1-\mathsf{P}_t)e_z) \\
& \quad + \mathsf{K}\delta_t 
\mathsf{P}_t \mathsf{K}\mathsf{P}_t (a_z^t + b_z^t - 2e_z) 
- \mathsf{K}\mathsf{P}_t \mathsf{K}
\delta_t\mathsf{P}_t (a_z^t + b_z^t - 2e_z). 
\endaligned 
\]
Using \eqref{c_515}, the right-hand side is calculated as 
\[
-\overline{(a_z^t(t) - b_z^t(t))}\mathsf{K}\delta_t 
- (a_z^t(t) + b_z^t(t))
\mathsf{K}\mathsf{P}_t \mathsf{K}\delta_t, 
\]
and $e_z$ in the second term on the left-hand side can be removed, 
since $(\partial/\partial x -iz)e_z=0$. 
Hence we obtain
\begin{equation} \label{c_517}
\aligned 
\left( 1- \mathsf{K}\mathsf{P}_t \mathsf{K}\mathsf{P}_t\right) & 
\left( \frac{\partial}{\partial x}-iz \right) (a_z^t + b_z^t) \\
& = 
-\overline{(a_z^t(t) - b_z^t(t))}\mathsf{K}\delta_t 
- (a_z^t(t) + b_z^t(t))
\mathsf{K}\mathsf{P}_t \mathsf{K}\delta_t. 
\endaligned
\end{equation}

On the other hand, we have 
\[
\phi^+ 
- \mathsf{K}\mathsf{P}_t \mathsf{K}\mathsf{P}_t\phi^+
= \mathsf{K}\delta_t - \mathsf{K}\mathsf{P}_t\mathsf{K}\delta_t, 
\qquad 
\phi^- 
- \mathsf{K}\mathsf{P}_t \mathsf{K}\mathsf{P}_t\phi^-
= \mathsf{K}\delta_t + \mathsf{K}\mathsf{P}_t\mathsf{K}\delta_t
\]
by \eqref{c_501} and \eqref{c_502}. 
By taking a linear combination of these two equations, 
\[
\aligned 
(1 - \mathsf{K}\mathsf{P}_t \mathsf{K}\mathsf{P}_t)&((b_0+ia_1) \phi^+ - (a_0+ib_1)\phi^-) \\
& = - \overline{(a-b)}\mathsf{K}\delta_t - (a+b)
\mathsf{K}\mathsf{P}_t\mathsf{K}\delta_t
\endaligned 
\]
for $a=a_0+ia_1$, $b=b_0+ib_1$ with $a_0,a_1,b_0, b_1 \in \R$. 
Taking $a$ and $b$ as $a_z^t(t)$ and $b_z^t(t)$, 
$((b_0+ia_1) \phi^+ - (a_0+ib_1)\phi^-)$ satisfies the same equation as 
\eqref{c_517}, so (O4) leads \eqref{c_513}, 
since $1 - \mathsf{K}\mathsf{P}_t \mathsf{K}\mathsf{P}_t
=
(1 - \mathsf{K}\mathsf{P}_t)
(1 + \mathsf{K}\mathsf{P}_t)
$. 
\end{proof}

\begin{proposition} \label{prop_5_9}
Let $t \in \R$ and $z \in \C_+$. 
Suppose that $\Vert \mathsf{K}[t] \Vert_{\rm op}<1$ 
and that (O3), (O4) are satisfied. 
Then the following equalities hold: 
\begin{equation} \label{c_518}
\frac{1}{2} \Bigl[
(a_z^t(t)+ b_z^t(t))-\overline{(a_z^t(t)-b_z^t(t))}
\Bigr] = e^{izt} - \int_{t}^{\infty} \overline{\phi^-(t,x)}e^{izx} \, dx, 
\end{equation}
\begin{equation} \label{c_519}
\frac{1}{2} \Bigl[
(a_z^t(t)+ b_z^t(t))+\overline{(a_z^t(t)-b_z^t(t))}
\Bigr] = e^{izt} + \int_{t}^{\infty} \overline{\phi^+(t,x)} e^{izx} \, dx. 
\end{equation}
\end{proposition}  
\begin{proof} It is sufficient to show that 
\begin{equation} \label{c_520}
(a_z^t(t)+b_z^t(t))
= 2e^{izt} + \int_{t}^{\infty} \overline{(\phi^+(t,x)-\phi^-(t,x))} e^{izx} \, dx, 
\end{equation} 
\begin{equation} \label{c_521}
\overline{(a_z^t(t)-b_z^t(t))}
=  \int_{t}^{\infty} \overline{(\phi^+(t,x)+\phi^-(t,x))} e^{izx} \, dx, 
\end{equation}
because 
\eqref{c_518} and \eqref{c_519} are obtained from 
the difference and sum of \eqref{c_520} and \eqref{c_521}, respectively. 
In this proof, we use the paring symbol 
$\langle f,g \rangle = \int f(x)\overline{g(x)} \, dx$ 
for simplification of the description, if the right-hand side makes sense. 
In particular, we describe the point evaluation of $f $ by 
$f(t)=\langle f , \delta_t \rangle$ and $\overline{f(t)}=\langle \delta_t, f \rangle$. 
Note that $\mathsf{K}$ is self-adjoint with respect to this paring:
\begin{equation} \label{eq_221216_1} 
\langle \mathsf{K}f,g \rangle = \langle \mathsf{K}g,f \rangle.
\end{equation}

First, we prove \eqref{c_520}. 
Taking the pairing of both sides of 
$(a_z^t + b_z^t-2\mathsf{P}_te_z) = 2(1-\mathsf{P}_t)e_z- \mathsf{K}\mathsf{P}_t (a_z^t - b_z^t)$ 
and 
$\mathsf{P}_t\mathsf{K}\delta_t - \mathsf{P}_t\mathsf{K}\mathsf{P}_t\phi_t^+ = \mathsf{P}_t\phi^+$
obtained from \eqref{c_514} and \eqref{c_501}, respectively, we obtain
\[
\langle (a_z^t + b_z^t-2\mathsf{P}_te_z) , 
\mathsf{P}_t(\mathsf{K}\delta_t - \mathsf{K}\mathsf{P}_t\phi^+) \rangle
= 
\langle  2(1-\mathsf{P}_t)e_z - \mathsf{K}\mathsf{P}_t (a_z^t - b_z^t), \mathsf{P}_t\phi^+\rangle. 
\]
Using the orthogonality of 
$\mathsf{P}_t$ and $(1-\mathsf{P}_t)$ on the right-hand side, 
\begin{equation} \label{c_522} 
\aligned 
\langle & (a_z^t + b_z^t -2\mathsf{P}_te_z), \mathsf{P}_t\mathsf{K}\delta_t \rangle
 - \langle (a_z^t + b_z^t -2\mathsf{P}_t), \mathsf{P}_t\mathsf{K}\mathsf{P}_t\phi^+ \rangle \\
& = 
- \langle \mathsf{K}\mathsf{P}_t (a_z^t - b_z^t), \mathsf{P}_t\phi^+\rangle. 
\endaligned 
\end{equation}
By a similar argument, we obtain 
\begin{equation} \label{c_523}
\aligned 
\langle & (a_z^t + b_z^t -2\mathsf{P}_te_z), \mathsf{P}_t\mathsf{K}\delta_t \rangle
 + \langle (a_z^t + b_z^t -2\mathsf{P}_t), \mathsf{P}_t\mathsf{K}\mathsf{P}_t\phi^- \rangle \\ 
& = 
- \langle \mathsf{K}\mathsf{P}_t (a_z^t - b_z^t), \mathsf{P}_t\phi^-\rangle
\endaligned 
\end{equation}
from \eqref{c_502} and \eqref{c_514}. 
On the other hand, paring \eqref{c_501} and the equality  
$(a_z^t - b_z^t) 
= 2 \, \mathsf{K}(1-\mathsf{P}_t)e_z - \mathsf{K}\mathsf{P}_t (a_z^t + b_z^t - 2\mathsf{P}_te_z)$ 
obtained from \eqref{c_515}, 
\begin{equation} \label{c_524}
\aligned 
\langle & \mathsf{P}_t\mathsf{K}\delta_t, (a_z^t - b_z^t) \rangle
 - \langle \mathsf{P}_t\mathsf{K}\mathsf{P}_t\phi^+ , (a_z^t - b_z^t) \rangle \\
& = 2
\langle  
\mathsf{P}_t\phi^+, \mathsf{K}(1-\mathsf{P}_t)e_z \rangle
- \langle \mathsf{P}_t\phi^+,  \mathsf{K}\mathsf{P}_t (a_z^t + b_z^t - 2\mathsf{P}_te_z) \rangle.
\endaligned 
\end{equation}
By a similar argument, we obtain 
\begin{equation} \label{c_525}
\aligned
\langle & \mathsf{P}_t\mathsf{K}\delta_t , (a_z^t - b_z^t) \rangle
 + \langle \mathsf{P}_t\mathsf{K}\mathsf{P}_t\phi^- , (a_z^t - b_z^t) \rangle \\
& = 2
\langle  
\mathsf{P}_t\phi^-, \mathsf{K}(1-\mathsf{P}_t)e_z \rangle
- \langle \mathsf{P}_t\phi^-,  \mathsf{K}\mathsf{P}_t (a_z^t + b_z^t - 2\mathsf{P}_te_z) \rangle
\endaligned 
\end{equation}
from \eqref{c_502} and \eqref{c_515}. 

On the other hand, we have 
\begin{equation} \label{c_526}
(a_z^t + b_z^t)(t) 
= \langle a_z^t + b_z^t, \delta_t \rangle
= 2e_z(t) - \langle \mathsf{P}_t\mathsf{K}\delta_t , (a_z^t - b_z^t) \rangle,
\end{equation}
\begin{equation} \label{c_527}
\overline{(a_z^t - b_z^t)(t)} = 2 \langle \delta_t, \mathsf{K}(1-\mathsf{P}_t)e_z \rangle 
- \langle  (a_z^t + b_z^t -2\mathsf{P}_t), \mathsf{P}_t\mathsf{K}\delta_t \rangle
\end{equation}
by the convention of the symbol, \eqref{c_514}, and \eqref{c_515}. Therefore, 
\begin{equation} \label{c_528}
\aligned 
(a_z^t +b_z^t)(t) & - \overline{(a_z^t-b_z^t)(t)} \\
 &= 2e_z(t) - 2 \langle \delta_t, \mathsf{K}(1-\mathsf{P}_t)e_z \rangle    \\
& \quad 
+ \langle  (a_z^t + b_z^t -2\mathsf{P}_te_z), \mathsf{P}_t\mathsf{K}\delta_t \rangle
- \langle \mathsf{P}_t\mathsf{K}\delta_t , (a_z^t - b_z^t) \rangle.  
\endaligned 
\end{equation}
The right-hand side is equal to 
\begin{equation} \label{c_529}
\aligned 
\, & 2e_z(t) - 2 \langle \delta_t, \mathsf{K}(1-\mathsf{P}_t)e_z \rangle   \\
& \quad 
- \frac{1}{2}\langle \mathsf{K}\mathsf{P}_t (a_z^t - b_z^t), \mathsf{P}_t(\phi^+ + \phi^-)\rangle \\
& \qquad 
+ \frac{1}{2}\langle (a_z^t + b_z^t -2\mathsf{P}_te_z), 
\mathsf{P}_t\mathsf{K}\mathsf{P}_t(\phi^+ - \phi^-) \rangle 
\\
& \quad 
- \langle  \mathsf{P}_t(\phi^+ + \phi^-), \mathsf{K}(1-\mathsf{P}_t)e_z \rangle \\
& \qquad + \frac{1}{2} \langle \mathsf{P}_t(\phi^+ + \phi^-),  
\mathsf{K}\mathsf{P}_t (a_z^t + b_z^t - 2\mathsf{P}_te_z) \rangle \\
& \quad 
- \frac{1}{2} \langle \mathsf{P}_t\mathsf{K}\mathsf{P}_t(\phi^+ - \phi^-) , (a_z^t - b_z^t) \rangle 
\endaligned 
\end{equation}
by using \eqref{c_522} and  \eqref{c_523} to the third term 
and \eqref{c_524} and \eqref{c_525} to the fourth term 
of the right-hand side. 
Using \eqref{eq_221216_1}, 
we find that the sum of the fourth and the sixth term of \eqref{c_529} is equal to 
$\langle (a_z^t + b_z^t -2\mathsf{P}_t), \mathsf{P}_t\mathsf{K}\mathsf{P}_t \phi^+  \rangle$ 
and that the sum of the third and seventh term of \eqref{c_529} is equal to 
$-\langle \mathsf{P}_t\mathsf{K}\mathsf{P}_t \phi^+ ,  (a_z^t - b_z^t) \rangle$. 
Therefore, \eqref{c_528} is 
\begin{equation} \label{c_530}
\aligned 
(a_z^t & +b_z^t)(t)  - \overline{(a_z^t-b_z^t)(t)} \\
& = 2e_z(t) - 2 \langle \delta_t, \mathsf{K}(1-\mathsf{P}_t)e_z \rangle  
- \langle  \mathsf{P}_t(\phi^+ + \phi^-), \mathsf{K}(1-\mathsf{P}_t)e_z \rangle  \\
& \quad 
- \langle \mathsf{P}_t\mathsf{K}\mathsf{P}_t \phi^+ ,  (a_z^t - b_z^t) \rangle 
+ \langle (a_z^t + b_z^t -2\mathsf{P}_te_z), \mathsf{P}_t\mathsf{K}\mathsf{P}_t \phi^+  \rangle.
\endaligned 
\end{equation}
The sum of the fourth and fifth term of the right-hand side is 
equal to 
$\langle \delta_t, \mathsf{K}\mathsf{P}_t (a_z^t + b_z^t -2\mathsf{P}_te_z) \rangle$ 
by \eqref{eq_221216_1} and \eqref{c_522}, 
and that is equal to  
$2 \langle \delta_t, \mathsf{K}(1-\mathsf{P}_t)e_z \rangle 
- \overline{(a_z^t - b_z^t)(t)}$ by \eqref{c_515}. 
As a result, we obtain 
\[
\aligned 
(a_z^t +b_z^t)(t)
& = 2e_z(t)  
- \langle  \mathsf{P}_t(\phi^+ + \phi^-), \mathsf{K}(1-\mathsf{P}_t)e_z \rangle
\endaligned 
\]
from \eqref{c_530}. The second term of the right-hand side 
is calculated as 
\[
\aligned 
\langle \mathsf{P}_t(\phi^+ + \phi^-), \mathsf{K}(1-\mathsf{P}_t)e_z \rangle
& = \langle (1-\mathsf{P}_t)e_z, \mathsf{K}\mathsf{P}_t(\phi^+ + \phi^-)  \rangle \\
& = - \langle (1-\mathsf{P}_t)e_z, (\phi^+ - \phi^-)  \rangle
\endaligned 
\]
by \eqref{eq_221216_1} and 
$(\phi^+ - \phi^-) + \mathsf{K}\mathsf{P}_t(\phi^++\phi^-)= 0$ 
obtained from \eqref{c_501} and \eqref{c_502}. Hence we obtain \eqref{c_520}. 

Next, we prove \eqref{c_521}. 
The following process looks similar to the proof of \eqref{c_520}, 
but actually different. We have 
\begin{equation} \label{c_531}
\aligned 
(a_z^t+b_z^t)(t) & + \overline{(a_z^t-b_z^t)(t)} 
 =  2e_z(t) + 2 \langle \delta_t, \mathsf{K}(1-\mathsf{P}_t)e_z \rangle \\
& \quad 
+ \frac{1}{2}\langle \mathsf{K}\mathsf{P}_t (a_z^t - b_z^t), \mathsf{P}_t(\phi^++\phi^-)\rangle \\
& \qquad 
- \frac{1}{2}\langle (a_z^t + b_z^t -2\mathsf{P}_te_z), \mathsf{P}_t\mathsf{K}\mathsf{P}_t(\phi^+-\phi^-) 
\\
& \quad- \langle  \mathsf{P}_t(\phi^++\phi^-), \mathsf{K}(1-\mathsf{P}_t)e_z \rangle \\
& \quad + \frac{1}{2} \langle \mathsf{P}_t(\phi^++\phi^-), 
\mathsf{K}\mathsf{P}_t (a_z^t + b_z^t - 2\mathsf{P}_te_z) \rangle \\
& \qquad - \frac{1}{2} \langle \mathsf{P}_t\mathsf{K}\mathsf{P}_t(\phi^+-\phi^-) , (a_z^t - b_z^t) \rangle
\endaligned 
\end{equation}
by using \eqref{c_524} and \eqref{c_525} 
on the second term of the right-hand side of \eqref{c_526}, 
and using \eqref{c_522} and \eqref{c_523} 
on the second term of the right-hand side of \eqref{c_527}. 
The right-hand side of \eqref{c_531} is equal to 
\begin{equation} \label{c_532}
\aligned 
2e_z(t) & + 2 \langle \delta_t, \mathsf{K}(1-\mathsf{P}_t)e_z \rangle 
- \langle  \mathsf{P}_t(\phi^++\phi^-), \mathsf{K}(1-\mathsf{P}_t)e_z \rangle \\
&  
+ \langle \mathsf{P}_t\mathsf{K}\mathsf{P}_t \phi^-,  (a_z^t - b_z^t) \rangle
 + \langle
 (a_z^t + b_z^t - 2\mathsf{P}_te_z), \mathsf{P}_t\mathsf{K}\mathsf{P}_t \phi^- \rangle 
\endaligned 
\end{equation}
by \eqref{eq_221216_1} as well as the proof of \eqref{c_520}. 

By taking the paring of 
$\mathsf{P}_t\phi^-$ and $(a_z^t - b_z^t) = 2 \mathsf{K}(1-\mathsf{P}_t)e_z 
- \mathsf{K}\mathsf{P}_t (a_z^t + b_z^t - 2\mathsf{P}_te_z)$ 
obtained from \eqref{c_515}, 
\[
\langle \mathsf{P}_t\phi^-, (a_z^t - b_z^t) \rangle 
= 2 \langle \mathsf{P}_t\phi^-, \mathsf{K}(1-\mathsf{P}_t)e_z \rangle 
- \langle \mathsf{P}_t\phi^-, \mathsf{K}\mathsf{P}_t (a_z^t + b_z^t - 2\mathsf{P}_te_z) \rangle. 
\]
Using \eqref{eq_221216_1} on the second term of the right-hand side, 
the equality becomes  
\[
\langle (a_z^t + b_z^t - 2\mathsf{P}_te_z), \mathsf{P}_t\mathsf{K}\mathsf{P}_t\phi^- \rangle 
= 2 \langle \mathsf{P}_t\phi^-, \mathsf{K}(1-\mathsf{P}_t)e_z \rangle 
- \langle \mathsf{P}_t\phi^-, (a_z^t - b_z^t) \rangle. 
\]
Substituting this in the last term of \eqref{c_532}, 
\[
\aligned 
(a_z^t +b_z^t)(t) & + \overline{(a_z^t-b_z^t)(t)} 
 =  
2e_z(t) + 2 \langle \delta_t, \mathsf{K}(1-\mathsf{P}_t)e_z \rangle  \\
& \quad- \langle  \mathsf{P}_t(\phi^+-\phi^-), \mathsf{K}(1-\mathsf{P}_t)e_z \rangle 
+ \langle \mathsf{P}_t\mathsf{K}\mathsf{P}_t \phi^- - \mathsf{P}_t\phi^-,  (a_z^t - b_z^t) \rangle. 
\endaligned 
\]
The last term of the right-hand side is equal to 
$- \langle \mathsf{K}\mathsf{P}_t(a_z^t - b_z^t), \delta_t   \rangle$ 
by \eqref{c_502} and \eqref{eq_221216_1}. 
Then 
$- \langle \mathsf{K}\mathsf{P}_t (a_z^t - b_z^t), \delta_t \rangle 
=
\langle (a_z^t + b_z^t), \delta_t \rangle 
-  2 \langle e_z, \delta_t \rangle$ 
by arranging the left-hand side of the equality 
$\langle (a_z^t + b_z^t-2\mathsf{P}_te_z), \delta_t \rangle 
-  2 \langle (1-\mathsf{P}_t)e_z, \delta_t \rangle 
= - \langle \mathsf{K}\mathsf{P}_t (a_z^t - b_z^t), \delta_t \rangle 
$
obtained from \eqref{c_514}.
Therefore, 
$
(a_z^t +b_z^t)(t) + \overline{(a_z^t-b_z^t)(t)} 
 =  
 2 \langle \delta_t, \mathsf{K}(1-\mathsf{P}_t)e_z \rangle  
- \langle  \mathsf{P}_t(\phi^+-\phi^-), \mathsf{K}(1-\mathsf{P}_t)e_z \rangle 
+  (a_z^t + b_z^t)(t), 
$
and thus 
\[
\overline{(a_z^t-b_z^t)(t)} 
=  
 2 \langle \delta_t, \mathsf{K}(1-\mathsf{P}_t)e_z \rangle 
- \langle  \mathsf{P}_t(\phi^+-\phi^-), \mathsf{K}(1-\mathsf{P}_t)e_z \rangle . 
\]
The second term of the right-hand side is equal to 
$\langle  (1-\mathsf{P}_t)e_z, \mathsf{K}\mathsf{P}_t(\phi^+-\phi^-)\rangle$, 
and it is further equal to $2 \langle  (1-\mathsf{P}_t)e_z, \mathsf{K}\delta_t \rangle 
- \langle  (1-\mathsf{P}_t)e_z, (\phi^+ + \phi^-) \rangle$, 
since $ \mathsf{K}\mathsf{P}_t(\phi^+-\phi^-)= 2\mathsf{K}\delta_t
- (\phi^+ + \phi^-)$ by
\eqref{c_501} and \eqref{c_502}. 
Hence, we obtain \eqref{c_521}. 
\end{proof}

\begin{proposition} \label{prop_5_10}
Let $t \in \R$ and $z \in \C_+$. 
Suppose that $\Vert \mathsf{K}[t] \Vert_{\rm op}<1$ 
and that (O3), (O4) are satisfied. Then the following equalities hold: 
\begin{equation} \label{c_533}
\aligned 
\frac{1}{2}\left[
e^{izt} + \int_{t}^{\infty} \phi^+(t,x)e^{izx}\,dx 
\right]
& = 
\frac{\Re(\Phi(t,t))}{\Re(\overline{\Psi(t,t)}\Phi(t,t))} \, \tilde{A}(t,z) \\
& \quad  
- \frac{\Im(\Psi(t,t))}{\Re(\overline{\Psi(t,t)}\Phi(t,t))} \cdot i \cdot (-i\tilde{B}(t,z)), 
\endaligned 
\end{equation}
\begin{equation} \label{c_534}
\aligned 
\frac{1}{2}\left[
e^{izt} - \int_{t}^{\infty} \phi^-(t,x)e^{izx}\,dx 
\right]
& = 
- \frac{\Im(\Phi(t,t))}{\Re(\overline{\Psi(t,t)}\Phi(t,t))} \cdot i \cdot \tilde{A}(t,z) \\
& \quad 
+ \frac{\Re(\Psi(t,t))}{\Re(\overline{\Psi(t,t)}\Phi(t,t))} (-i \tilde{B}(t,z)). \hspace{12pt}
\endaligned 
\end{equation}
\end{proposition} 
\begin{proof} By definition \eqref{c_205}, \eqref{c_404}, 
and the derivative rule for $\mathsf{F}$, 
\[
2 \tilde{A}(t,z) 
= 
\Psi(t,t)e^{izt} + \int_{t}^{\infty}\frac{\partial}{\partial x}\Psi(t,x)e^{izx} \, dx, 
\]
\[
-2i \tilde{B}(t,z) 
= 
\Phi(t,t)e^{izt} + \int_{t}^{\infty}\frac{\partial}{\partial x}\Phi(t,x)e^{izx} \, dx. 
\]
Combining these with \eqref{c_503} and \eqref{c_504},  
\[
\aligned 
\int_{t}^{\infty} \phi^+(t,x)e^{izx}\,dx 
&= \frac{\Re(\Phi(t,t))}
{\Re(\overline{\Psi(t,t)}\Phi(t,t))}
(2 \tilde{A}(t,z)) \\ 
&\quad -
\frac{\Im(\Psi(t,t))}
{\Re(\overline{\Psi(t,t)}\Phi(t,t))}
(2\tilde{B}(t,z) ) - e^{izt}, 
\endaligned 
\]
\[
\aligned 
\int_{t}^{\infty} \phi^-(t,x) e^{izx}\,dx
& =\frac{\Im(\Phi(t,t))}
{\Re(\overline{\Psi(t,t)}\Phi(t,t))} (2i \tilde{A}(t,z)) \\
& \quad  -
\frac{\Re(\Psi(t,t))}
{\Re(\overline{\Psi(t,t)}\Phi(t,t))} (-2i \tilde{B}(t,z)) + e^{izt}. 
\endaligned 
\]
Hence we obtain \eqref{c_533} and \eqref{c_534}.  
\end{proof}

%
\subsection{Proof of Theorem \ref{thm_2_2}} \label{section_5_2}
%

The first equality of \eqref{c_210} is shown 
in the same way as in the proof of \cite[Theorem 4.1]{Su19_1}. 
To prove the second equality of \eqref{c_210}, 
we calculate the Fourier transform of $Y_z^t$. 
We have
\[
\aligned 
-i(z+w) & \int_{0}^{\infty} Y_z^t(x) e^{iwx}\,dx 
= -i(z+w) \int_{t}^{\infty} \frac{1}{2}(a_z^t(x)+b_z^t(x)) e^{iwx}\,dx \\
& = \int_{t}^{\infty} \frac{1}{2}\left(\frac{\partial}{\partial x} -iz \right) (a_z^t(x)+b_z^t(x)) e^{iwx}\,dx 
 + \frac{1}{2}(a_z^t(t)+b_z^t(t)) e^{iwt}  
\endaligned 
\]
by \eqref{c_510} for $z, w \in \C_+$. The right-hand side is calculated as 
\[
\aligned 
\frac{1}{2} & b_z^t(t) e^{iwt} + \frac{1}{4}
\left( (a_z^t(t)+b_z^t(t)) 
-\overline{(a_z^t(t)-b_z^t(t))}
\right) 
\int_{t}^{\infty}
\phi^+(t,x) e^{iwx} \, dx  \\
& 
+ \frac{1}{2}a_z^t(t) e^{iwt} - \frac{1}{4}
\left(
 (a_z^t(t)+ b_z^t(t)) + \overline{(a_z^t(t)-b_z^t(t))}
\right) 
\int_{t}^{\infty}
\phi^-(t,x) e^{iwx} \, dx \\
& = \frac{1}{4}
\left[ (a_z^t(t)+b_z^t(t)) 
-\overline{(a_z^t(t)-b_z^t(t))}
\right] \left[ e^{iwt} + \int_{t}^{\infty}
\phi^+(t,x) e^{iwx} \, dx \right] \\ 
& \quad 
+ \frac{1}{4}
\left[
 (a_z^t(t)+ b_z^t(t)) + \overline{(a_z^t(t)-b_z^t(t))}
\right] \left[ e^{iwt} - \int_{t}^{\infty}
\phi^-(t,x) e^{iwx} \, dx \right] 
\endaligned 
\]
by \eqref{c_513}. Therefore,  
\[
\aligned 
-i(z+w) &  \int_{0}^{\infty} Y_z^t(x) e^{iwx}\,dx \\
& = \frac{1}{2}
\left[ e^{izt} - \int_{t}^{\infty} \overline{\phi^-(t,x)}e^{izx} \, dx
\right] \left[ e^{iwt} + \int_{t}^{\infty}
\phi^+(t,x) e^{iwx} \, dx \right] \\ 
& \quad 
+ \frac{1}{2}
\left[
 e^{izt} + \int_{t}^{\infty} \overline{(\phi^+(t,x))} e^{izx} \, dx
\right] \left[ e^{iwt} - \int_{t}^{\infty}
\phi^-(t,x) e^{iwx} \, dx \right] 
\endaligned 
\]
by \eqref{c_518} and \eqref{c_519}. Hence, 
\[
\frac{1}{2\pi}\int_{0}^{\infty} Y_z^t(x) e^{iwx}\,dx 
=  
\frac{
\overline{\tilde{A}(t,-\bar{z})}\tilde{B}(t,w)
-\overline{\tilde{B}(t,-\bar{z})}\tilde{A}(t,w)
}
{\pi(z+w)}
\]
by \eqref{c_533} and \eqref{c_534}. 
This implies the second equality of \eqref{c_210} 
for $z, w \in \C_+$, 
because 
\[
\langle Y_w^t,Y_z^t \rangle 
= \int_{t}^{\infty} Y_w^t(x) Y_{-\bar{z}}^t(x) \, dx 
= \int_{t}^{\infty} Y_{-\bar{z}}^t(x) e^{iwx}\, dx
\]
by definition of the vector $Y_z^t$. 
Equality \eqref{c_210} extends to $z,w \in \C_+ \cup D$ 
by analytic continuation, since the second and the third term of \eqref{c_210} 
extend to $z,w \in \C_+ \cup D$ by 
Lemma \ref{lem_5_4} and Theorem \ref{thm_2_1}, respectively. 
\hfill $\Box$

%
\subsection{Proof of Theorem \ref{thm_2_3}}
%

In preparation for the proof of Theorems \ref{thm_2_3} and \ref{thm_2_4}, 
we state one result related to the condition (O8). 

\begin{proposition} \label{prop_5_11}
Let $u\in U_{\rm loc}^1(\R)$. 
Suppose that (O2), (O3), (O4), (O5), (O6) are satisfied. 
Then, 
\begin{enumerate}
\item[(1)]  if $t_0 = \infty$, $\lim_{t \to t_0}\Vert Y_z^t \Vert_{L^2(\R)}=0$ 
for every $z \in \C_+$;  
\item[(2)]  if $t_0<\infty$, then $\mathcal{V}_{t_0}(u)=\{0\}$ 
if and only if  $\lim_{t \to t_0}\Vert Y_z^t \Vert_{L^2(\R)}=0$
for almost every $z \in \C_+$. 
\end{enumerate}
\end{proposition}
\begin{proof} 
(1) can be proved in the same way as in \cite[Section 4.3]{Su19_1}, 
but here is a simpler proof.  
From the decomposition 
$(1-\mathsf{P}_t)e_z = Y_z^t + J_z^t$ with  
$Y_z^t \in \mathcal{V}_t$ and $J_z^t \in \mathcal{V}_t^\perp$, 
we have
$\Vert Y_z^t \Vert_{L^2(\R)} 
\leq 
\Vert (1-\mathsf{P}_t)e_z \Vert_{L^2(\R)} 
= (2\,\Im(z))^{-1/2}\,e^{-t\cdot \Im(z)}$. 
Hence $\lim_{t \to \infty}\Vert Y_z^t \Vert_{L^2(\R)}=0$.

We prove (2). 
If $t<s$, $(f,Y_z^t)=(\mathsf{F}f)(z)=(f,Y_z^s)$ for every $f\in\mathcal{V}_s(u)$, 
since $\mathcal{V}_t(u) \supset \mathcal{V}_s(u)$. 
Therefore, $(f,Y_z^t-Y_z^s)=0$ for every $f\in\mathcal{V}_s(u)$, 
that is, $Y_z^t-Y_z^s \in \mathcal{V}_s(u)^{\perp}$. 
Thus, $Y_z^t = (Y_z^t-Y_z^s)+Y_z^s$ is an orthogonal decomposition. 
In particular, 
$\Vert Y_z^t \Vert^2 = \Vert Y_z^t-Y_z^s \Vert^2+ \Vert Y_z^s \Vert^2$, 
so $\Vert Y_z^t \Vert$ is non-increasing with respect to $t$. 
Hence $\lim_{t \to t_0} \Vert Y_z^t \Vert$ exists. 
On the other hand, $\Vert Y_z^t \Vert^2 - \Vert Y_z^s \Vert^2 
= \Vert Y_z^t-Y_z^s \Vert^2 \geq 0$ 
shows that the convergence of the norm $\Vert Y_z^t \Vert$ implies the convergence of $Y_z^t$. 
Hence $\lim_{t \to t_0} Y_z^t$ exists in $L^2$ sense. 
Now we suppose $t<t_0$ and put $Y=\lim_{t \to t_0} Y_z^t$. 
Then $\langle f, \overline{Y} \rangle = \mathsf{F}f(z)$ 
for every $f \in \mathcal{V}_{t_0}(u)$, 
since $\langle f, \overline{Y_z^t} \rangle = \mathsf{F}f(z)$. 
Hence $Y=Y_z^{t_0}$ and (2) holds. 
\end{proof}

We have 
\[
\aligned
2\pi i(\bar{z}-w) 
j(t;z,w) 
&= 
\overline{(\tilde{A}(t,z) - i \tilde{B}(t,z))}(\tilde{A}(t,w) - i \tilde{B}(t,w)) \\
& \quad -\overline{(\tilde{A}(t,z) + i \tilde{B}(t,z))}(\tilde{A}(t,w) + i \tilde{B}(t,w)))
\endaligned 
\]
by direct calculation. Thus 
\[
j(t;z,z)=
\frac{|\tilde{A}(t,z) - i \tilde{B}(t,z)|^2 -|\tilde{A}(t,z) + i \tilde{B}(t,z)|^2}{2\pi i(\bar{z}-z)} 
\]
for $z \in \C_+$. On the other hand, we obtain 
\begin{equation*}
j(t;z,z)  
 = \left( \lim_{t \to t_0} j(t,z,z) \right) 
+ \frac{1}{\pi}\int_{t}^{t_0} \,
\begin{bmatrix}
\tilde{A}(u,w) & \tilde{B}(u,w)
\end{bmatrix}
H(u)
\overline{
\begin{bmatrix}
\tilde{A}(u,z) \\ \tilde{B}(u,z)
\end{bmatrix}
} \, du 
\end{equation*}
from the first order system \eqref{c_207} as in the proof of \cite[(2.40)]{Su19_1}.  
The first equality of \eqref{c_210} shows that 
$\lim_{t \to t_0} j(t,z,z)= (1/2\pi)\lim_{t \to t_0}\Vert Y_z^t \Vert^2$, 
and the right-hand side exists by the proof of Proposition \ref{prop_5_11}. 
Hence $j(t;z,z) \geq 0$ for $z \in \C_+$ by (O7), which implies  
$|\theta(t,z)| \leq 1$ for $z \in \C_+$ by definition \eqref{c_211}. 
On the other hand, 
\begin{equation*} 
\theta(t,z)^\sharp=\frac{\tilde{A}(t,z) - i \tilde{B}(t,z)}{\tilde{A}(t,z) + i \tilde{B}(t,z)} 
= \frac{1}{\theta(t,z)}
\end{equation*}
as a function of $z$ by definition \eqref{c_211} and \eqref{c_206}. 
Thus, $|\theta(t,z)|=1$ for real $z$.  
Hence $\theta(t,z)$ is inner. 
For an inner function $\theta$, 
the reproducing kernel of the model space $\mathcal{K}(\theta)$ is 
$(1/2\pi i)(1-\overline{\theta(z)}\theta(w))/(\bar{z}-w)$. 
We confirm that the reproducing kernel of $\mathsf{F}(\mathcal{V}_t(u))$ 
equals to the reproducing kernel of $\mathcal{K}(\theta(t,z))$ 
by direct calculation. 
\hfill $\Box$

%
\subsection{Analytic properties of 
$A(t,z)$ and $B(t,z)$} \label{section_5_5}
%

In the cases of $u(z)=M^\sharp(z)/M(z)$, 
we defined $A(t,z)$, $B(t,z)$, $E(t,z)$ by \eqref{c_212}. 
Then they are entire functions 
satisfying 
\[
A(t,z)= A^\sharp(t,z), \quad B(t,z)= B^\sharp(t,z)
\] 
by \eqref{c_206} and Theorem \ref{thm_2_1}. 
Therefore, $E^\sharp(t,z)=A(t,z) +i B(t,z)$ and 
\[
A(t,z) = \frac{1}{2}(E(t,z) + E^\sharp(t,z)), \quad 
B(t,z) = \frac{i}{2}(E(t,z) - E^\sharp(t,z)). 
\]

If $M^\sharp(z)= \varepsilon M(-z)$ for a sign $\varepsilon \in \{\pm 1\}$, 
then $u$ is symmetric, 
and therefore $\Phi$ and $\Psi$ are real-valued. Hence, 
\[
\aligned 
E^\sharp(t,z)
& = M^\sharp(z)\frac{iz}{2}\int_{t}^{\infty} (\Phi(t,x)+\Psi(t,x)) e^{-izx} \,dx \\
& = \varepsilon M(-z)\frac{-i(-z)}{2}\int_{t}^{\infty} (\Phi(t,x)+\Psi(t,x)) e^{i(-z)x} \,dx 
= \varepsilon E(t,-z). 
\endaligned 
\]
Therefore, $A(t,z)$ is even and $B(t,z)$ is odd 
if $\varepsilon=+1$, and $A(t,z)$ is odd and $B(t,z)$ is even 
if $\varepsilon=-1$. 

%
\subsection{Conformity of the axiom of de Branges spaces}
%

\begin{proposition} \label{prop_5_12} 
Suppose that $u(z)=M^\sharp(z)/M(z)$ 
for some meromorphic function $M(z)$ on $\C$ 
such that it is holomorphic on $\C_+ \cup \R$ and has no zeros in $\C_+$.  
Further, suppose that  (O2), (O3), (O4), (O5), (O6) are satisfied. 
Then $M(z)\mathsf{F}(\mathcal{V}_t(u))$ is a de Branges space 
for every $t  < t_0$. 
\end{proposition}
\begin{proof} 
We show that $\mathcal{H}:=M(z)\mathsf{F}(\mathcal{V}_t(u))$ 
is a Hilbert space consisting of entire functions and 
satisfies the axiom of the de Branges spaces: 
\begin{enumerate}
\item[(dB1)] For each $z \in \C\setminus \R$ the point evaluation $\Phi \mapsto \Phi(z)$ is a continuous linear functional on $\mathcal{H}$; 
\item[(dB2)] If $\Phi \in \mathcal{H}$, $\Phi^\sharp$ belongs to $\mathcal{H}$ and $\Vert \Phi \Vert_{\mathcal{H}} = \Vert \Phi^\sharp \Vert_{\mathcal{H}}$;  
\item[(dB3)] If $w \in \C \setminus \R$, $\Phi \in \mathcal{H}$ and $\Phi(w)=0$, 
\begin{equation*}
\frac{z-\bar{w}}{z-w}\Phi(z) \in \mathcal{H} \quad \text{and} \quad 
\left\Vert \frac{z-\bar{w}}{z-w}\Phi(z) \right\Vert_{\mathcal{H}} = \Vert \Phi \Vert_{\mathcal{H}},
\end{equation*}
\end{enumerate}
where the Hilbert space structure is the one induced from $\mathcal{V}_t(u)$ 
that is equivalent to $\langle F,G\rangle_{\mathcal H}
= \int_{\R} F(z)\overline{G(z)}|M(z)|^{-2}dz$ for $F, G \in \mathcal{H}$. 

Let $\Phi(z)=M(z)(\mathsf{F}f)(z) \in \mathcal{H}$ with $f \in \mathcal{V}_t(u)$. 
First, we prove that $\mathcal{H}$ consists of entire functions. 
We see that $\Phi(z)$ is holomorphic on $\C_+$ by $f \in L^2(t,\infty)$
and is holomorphic on $\C_-$ by 
$\Phi(z)=M^\sharp(z)(\mathsf{F}\mathsf{J}_\sharp \mathsf{K}f)(z)$ 
and $\mathsf{J}_\sharp \mathsf{K}f \in L^2(-\infty,-t)$. 
Moreover, 
$\lim_{z \to x}(\mathsf{F}f)(z)=(\mathsf{F}f)(x)$ and 
$\lim_{z \to x}(\mathsf{F}\mathsf{J}_\sharp \mathsf{K}f)(z)
=\lim_{z \to x}(\mathsf{F}\mathsf{K}f)^\sharp(z)=u^\sharp(x)(\mathsf{F}f)(x)$ 
for almost all $x \in \R$, 
where $z$ is allowed to tends to $x$ non-tangentially 
from $\C_+$ and $\C_-$, respectively. 
Hence $\Phi(z)$ is also holomorphic in a neighborhood of each point of $\R$. 

We confirm (dB1). 
For $z \in \C_+$, $\Phi \mapsto \Phi(z)=M(z)\int_{t}^{\infty}f(x)e^{izx}dx$ is a continuous linear form. 
On the other hand, for $z \in \C_-$, $\Phi \mapsto \Phi(z)=M^\sharp(z)
\int_{-\infty}^{-t} \overline{(\mathsf{K}f)(-x)}e^{izx}\,dx$ is a continuous linear functional. 
(Moreover, for $z \in \R$, the continuity follows from the Banach-Steinhaus theorem.) 

We confirm (dB2). 
We have  
$\Phi^\sharp(z)
= M(z)(\mathsf{F}\mathsf{K}f)(z)$. 
Since $\mathsf{K}f \in \mathcal{V}_t(u)$, $\Phi^\sharp$ belongs to $\mathcal{H}$. 
Since $\mathsf{K}$ is isometric, the equality of norms in (dB2) holds. 

We confirm (dB3). 
The equality of norms in (dB3) is trivial by the definition of 
the norm of $\mathcal{H}$. 
From (dB2), it is sufficient to show only the case of $w \in \C_+$. 
Suppose that $\Phi(w)=0$ for $w \in \C_+$. 
Then $(\mathsf{F}f)(w)=0$, since $M(z)$ has no zeros on $\C_+$. 
We put $f_w(x)=f(x) - i(w-\bar{w}) \int_{0}^{x-t} f(x-y) e^{-iwy} dy$.   
Then we easily find that $f_w \in L^2(t,\infty)$ 
and $(\mathsf{F}f_w)(z)=((z-\bar{w})/(z-w))(\mathsf{F}f)(z)$ 
for $z \in \C_+$. 
Hence we complete the proof if it is shown that 
$\mathsf{K}f_w$ has support in $[t,\infty)$, 
since $\mathsf{K}f_w \in L^2(\R)$ by $f_w \in L^2(t,\infty)$. 
We put 
$g_w(x)=(\mathsf{K}f)(x) - i(\bar{w}-w) \int_{0}^{x-t} (\mathsf{K}f)(x-y) e^{-i\bar{w}y} dy$. 
Then $g_w$ has support in $[t,\infty)$ by $\mathsf{K}f \in L^2(t,\infty)$ 
and
$(\mathsf{F}g_w)(z)=((z-w)/(z-\bar{w}))(\mathsf{F}\mathsf{K}f)(z)
=(\mathsf{F}\mathsf{K}f_w)(z)$ for $z \in \C_+$.
Hence $g_w =\mathsf{K}f_w$ and the proof is completed. 
\end{proof}

%
\subsection{Proof of Theorem \ref{thm_2_4}}
%

As mentioned in Section \ref{section_5_5}, 
$A(t,z)$, $B(t,z)$, $E(t,z)$ are entire functions 
with the assumptions of Theorem \ref{thm_2_4}. 
Also, they satisfy the system of differential equation \eqref{c_101} 
for $t<t_0$ and $z \in \C$ by Theorem \ref{thm_2_1}. 
Further, we have 
$\Theta(t,z)=E^\sharp(t,z)/E(t,z)$ 
for $\Theta(t,z)$ defined by \eqref{c_211}, 
since $E^\sharp(t,z)=A(t,z)+iB(t,z)$. 
Thus $E(t,z) \in \overline{\mathbb{HB}}$ by Theorem \ref{thm_2_3} 
and therefore the de Branges space $\mathcal{H}(E(t,z))$ is defined. 
On the other hand,  $M(z)\mathsf{F}(\mathcal{V}_t(u))$ 
is also a de Branges space by Proposition \ref{prop_5_12}.  
Let $J(t;z,w)$ be the reproducing kernel of $M(z)\mathsf{F}(\mathcal{V}_t(u))$. 
Then, 
\[
J(t;z,w) = \overline{M(z)}M(w)j(t;z,w)
= 
\frac{ 
\overline{A(t,z)}B(t,w)  
- A(t,w)\overline{B(t,z)}
}{\pi(w-\bar{z})} 
\]
for every $t<t_0$ by Theorem \ref{thm_2_2}. 
The right-hand side is nothing but the reproducing kernel of $\mathcal{H}(E(t,z))$ 
(\cite[Section 3.2]{Su19_1}). 
Hence $M(z)\mathsf{F}(\mathcal{V}_t(u))=\mathcal{H}(E(t,z))$ 
for every $t<t_0$. 
To conclude that $H(t)$ on $[t_1,t_0)$ 
is the structure Hamiltonian of $\mathcal{H}(E(t_1,z))$, 
it remains to show $\lim_{t \to t_0} J(t;z,w)=0$, 
but this follows from (O8) by Proposition \ref{prop_5_11}. 
\hfill $\Box$

%
\section{Proof of results in Section \ref{section_2_3}} \label{section_6}
%

%
\subsection{Properties of $\mathsf{K}$ for an inner function $u$} \label{section_6_1}
%

To describe the properties of the operator $\mathsf{K}=\mathsf{K}_u$ 
when $u$ is an inner function in $\C_+$, 
we recall the following result on inner functions (\cite[Theorems 1.1 and 1.2]{QXYYY09}):  

\begin{proposition} \label{prop_6_1}
A unimodular function $u$ in $L_{\rm loc}^1(\R)$ 
is the nontangential limit of an inner function $\theta$ in $\C_+$ 
if and only if the tempered distribution 
$k=\mathsf{F}^{-1}u$ has support in $[0, \infty)$. 
\end{proposition}

Therefore, if $u$ is an inner function in $\C_+$, $\mathsf{K}f$ has support in $[-t,\infty)$ 
for every $f \in \mathsf{P}_t S'(\R)$, since  
\[
\aligned 
{\rm supp}\,\mathsf{K}f = 
{\rm supp}\,(k \ast \mathsf{J}_\sharp f) & \subset 
\overline{{\rm supp}\, k 
+
{\rm supp}\,\mathsf{J}_\sharp f} \\
& \subset 
\overline{[0,\infty)+[-t,\infty)}= [-t,\infty). 
\endaligned 
\]

\begin{proposition} \label{prop_6_2}
Suppose that $u$ is an inner function $\theta\not=1$ in $\C_+$. 
Then, 
\begin{enumerate}
\item[(1)]  $\mathsf{K}[t]=0$ as an operator for nonpositive $t$. In particular, (O1) holds; 
\item[(2)]  $H(t)$ defined by \eqref{c_208} and \eqref{c_209} is the identity matrix for all negative $t$:  
\item[(3)]  $\mathcal{V}_t(u)\not=\{0\}$ for nonpositive $t$. In particular, (O5) holds;  
\item[(4)]  $\mathsf{F}(\mathcal{V}_0(u))=\mathcal{K}(\theta)$. 
In particular, if $\theta=E^\sharp/E$ for some $E \in \mathbb{HB}$, 
we have $E(z)\mathsf{F}(\mathcal{V}_0(u))=\mathcal{H}(E)$. 
\end{enumerate}
\end{proposition}
\begin{proof}
For $f \in L^2(-\infty,t)$, $\mathsf{K}f$ belongs to $L^2(-t,\infty)$ 
by Proposition \ref{prop_6_1}, and therefore $\mathsf{K}[f]f=0$ for 
all $f \in L^2(-\infty,t)$, that is, (1) holds. 
For negative $t$, \eqref{c_203} and \eqref{c_204} are easily solved as 
$\Phi = 1- \mathsf{K}\mathsf{P}_t1$ and $\Psi = 1+ \mathsf{K}\mathsf{P}_t1$ 
by (1). Here $\mathsf{K}\mathsf{P}_t1$ has support in $[-t,\infty)$, 
thus $\Phi(t,t)=\Psi(t,t)=1$, which implies (2). 

To prove $\mathcal{V}_0(u)\not=\{0\}$, 
it is sufficient to show that $\mathsf{K}L^2(-\infty,0)$ 
is a proper subspace of $L^2(0,\infty)$ by \eqref{c_507}.  
The latter is true because 
the space of Fourier transforms 
$\mathsf{F}(\mathsf{K}L^2(-\infty,0))=\theta H^2(\C_+)$ 
is a proper subspace of $H^2(\C_+)=\mathsf{F}(L^2(0,\infty))$.  
For negative $t$, we have 
$\mathsf{K}L^2(-\infty,t) \subset L^2(-t,\infty) \subsetneq L^2(t,\infty)$. 
Therefore, $\mathcal{V}_t(u)^\perp$ is a proper subspace of $L^2(\R)$, 
and hence (3) holds. 

To prove (4), the proof of \cite[Lemma 4.1]{Su19_1} can be applied by replacing 
$\Theta(-z)$ and $(\mathsf{F}f)(-z)$, etc. in \cite{Su19_1} 
with $\theta^\sharp(z)$ and $(\mathsf{F}f)^\sharp(z)$, etc. in this paper. 
The difference on the definition of the operator $\mathsf{K}$ does not affect the argument of the proof. 
\end{proof}

For any of (O2), (O3), (O4), (O6), (O7), (O8), 
we do not know whether it holds for a general inner function. 
However, for (O6) and (O7), we can give sufficient conditions as in Section \ref{section_7}. 
As in the example in Section \ref{section_3_2}, 
(O8) may not hold in general even if $u$ is a meromorphic inner function. 

%
\subsection{Proof of Theorem \ref{thm_2_5}}
%

Note that $\tilde{A}(0,z)$ and $\tilde{B}(0,z)$ are defined, 
since $\mathsf{K}[0]=0$ by the proof of Proposition \ref{prop_6_2}. 
We have $
(1-\mathsf{P}_0)\Phi  + \mathsf{P}_0\Phi 
+ \mathsf{K}\mathsf{P}_0 \Phi = (1-\mathsf{P}_0)1 + \mathsf{P}_0 1
$ by \eqref{c_203}. 
Acting $\mathsf{P}_0$ to both sides gives $\mathsf{P}_0 \Phi=\mathsf{P}_0 1$, 
since  $\mathsf{K}\mathsf{P}_0\Phi$ has support in $[0,\infty)$. 
Thus 
$(1-\mathsf{P}_0)\Phi  = (1-\mathsf{P}_0)1- \mathsf{K}\mathsf{P}_01$. 
We calculate $\mathsf{K}\mathsf{P}_01$. For $g \in S(\R)$, 
if we write $G(z)=(\mathsf{F}g)(z)$, 
\[
\aligned 
\langle \mathsf{K}\mathsf{P}_0 1, g \rangle
& =  \langle \mathsf{K}g, \mathsf{P}_0 1 \rangle
 = \int \left(  \frac{1}{2\pi}\int_{\Im z=0} u(z)G^\sharp(z) e^{-izx} dz \right) \mathsf{P}_0 1(x) \, dx \\
& = \int \left(  \frac{1}{2\pi}\int_{\Im z=\delta>0} 
\theta(z)G^\sharp(z) e^{-izx} dz \right) \mathsf{P}_0 1(x) \, dx, 
\endaligned 
\]
since $u$ is bounded in $\C_+ \cup \R$. 
Then, we have 
\[
\aligned 
\mathsf{K}\mathsf{P}_0 1(x) 
= u(0)(1- \mathsf{P}_0)1(x)
+ 
\frac{1}{2\pi}\int_{\Im(z)=\delta>0} \frac{\theta(z)-u(0)}{-iz} \, e^{-izx} dz
\endaligned 
\]
in a way similar to the proof of Proposition \ref{prop_4_1}. 
Because the second term of the right-hand side belongs to $L^2(\R)$, 
$(\mathsf{F}(1-\mathsf{P}_0)\mathsf{K}\mathsf{P}_0 1)(z)$ is defined for $z\in \C_+$. 
Hence $(\mathsf{F}(1-\mathsf{P}_0)\Phi)(z)$ is defined for $z\in \C_+$. Therefore, 
\[
\aligned 
\frac{2}{-iz}(-i\tilde{B}(0,z))
= 
\mathsf{F}(1-\mathsf{P}_0)\Phi  
& = \mathsf{F} (1-\mathsf{P}_0)1 - \mathsf{F}\mathsf{K}\mathsf{P}_01 
 = \frac{1}{-iz}(1 - \theta(z)). 
\endaligned 
\]
The case of $\tilde{A}(0,z)$ is shown by a similar argument. 
\hfill $\Box$

%
\section{Complementary results} \label{section_7}
%

%
\subsection{A sufficient condition for (O6)}
%

\begin{proposition} \label{prop_7_1}
Suppose that $u$ is an inner function $\theta$ in $\C_+$ and continuous on $\R$. 
Then $\mathsf{K}[t]$ is compact for all $t \in \R$. 
\end{proposition}
\begin{remark}
Even if $u$ is not an inner function, $\mathsf{K}[t]$ can be a compact operator. 
For example, $u(z)=\Gamma(1/2+iz)/\Gamma(1/2-iz)$ is not an inner function, 
but $\mathsf{K}[t]$ is compact for every $t\in \R$ 
because we can check that the kernel $k(x)=e^{x/2}J_0(2e^{x/2})$ 
satisfies the Hilbert--Schmidt condition on $(-\infty,t] \times (-\infty,t]$. 
(Also, it is proved that $\mathsf{K}[t]$ is a limit of finite rank operators in \cite[Section 5]{Burnol2011}).   
\end{remark}
\begin{proof}
Only the case of positive $t$ needs to be proved by Proposition \ref{prop_6_2} (1). 
We find that $\mathsf{K}[t](L^2(-\infty,-t))=\{0\}$  
and $\mathsf{K}[t](L^2(-t,t)) \subset L^2(-t,t)$ by Proposition \ref{prop_6_1}. 
Therefore, it suffices to prove that the restriction $\left.\mathsf{K}[t]\right|_{L^2(-t,t)}$ is compact. 
For $f \in L^2(-t,t)$, 
\[
(\mathsf{K}[t]f)(x)
= \mathbf{1}_{[0,2t]}(x-t)\int_{0}^{2t}k((x-t)-y+2t)\overline{f(-y+t)}dy.
\]
Therefore, the restriction of $\mathsf{K}[t]$ to $L^2(-t,t)$ 
is a composition of the translation $f(x) \mapsto f(x-t)$, 
inversion $f(x) \mapsto f(-x)$, conjugation $f(x) \mapsto \overline{f(x)}$ 
and the operator $\mathsf{F}^{-1}A_\phi\mathsf{F}$ 
defined by $A_\phi F := \mathsf{Q}_{2t}\mathsf{M}_\phi F$, where 
\[
\phi(z) = \int_{0}^{\infty} k(x+2t)e^{izx}dx=e^{-2itz}\theta(z)-\int_{-2t}^{0}k(x+2t)e^{izx}dx
\]
and $\mathsf{Q}_{2t}$ is the projection from $L^2(0,\infty)$ 
to the Paley--Wiener space $\mathsf{F}L^2(0,2t)$. 
That is, $A_\phi$ is the truncated Toeplitz operator on $\mathsf{F}L^2(0,2t)$. 
Note that $\int_{-2t}^{0}k(x+2t)e^{izx}dx$ is entire, since $k$ is a tempered distribution, 
which is a higher derivative of a continuous function.  
Then $A_\phi$ is compact if $\theta$ is continuous on $\R$ 
by \cite[Theorem 5.1]{MR3589670} (see also \cite[Remark 3.5]{MR2679022}). 
Hence the restriction $\left.\mathsf{K}[t]\right|_{L^2(-t,t)}$ is compact. 
\end{proof}

\begin{proposition} \label{prop_7_3}
Suppose that $u$ is an inner function $\theta$ in $\C_+$ and is continuous on $\R$,  
and there are no entire functions $F$ and $G$ 
of exponential type such that $\theta =G/F$. 
Then $\Vert \mathsf{K}[t] \Vert_{\rm op} <1$ for all $t \in \R$. 
In particular, $t_0=\infty$ and (O6) is satisfied. 
\end{proposition}
\begin{proof} 
Only the case of positive $t$ needs to be proved by Proposition \ref{prop_6_2} (1). 
Let $t>0$. 
By applying the argument in the proof of \cite[Theorem 5.2]{Su19_1}, 
it is shown that $1$ is not an eigenvalue of $\mathsf{K}[t]$, 
since differences in the definition and properties of $\mathsf{K}$ 
do not affect the argument.
Therefore, according to the general theory of antilinear operators 
(\cite{MR2749452, Uhlmann}), 
if $|\lambda|=1$, then $\lambda$ is not an eigenvalue of $\mathsf{K}[t]$. 
Since $\Vert \mathsf{K}[t]\Vert_{\rm op} \leq 1$, 
$\Vert \mathsf{K}[t] \Vert_{\rm op} \not=1$ implies 
$\Vert \mathsf{K}[t] \Vert_{\rm op} <1$. 
Suppose that $\Vert \mathsf{K}[t] \Vert_{\rm op}=1$. 
Then $\Vert \mathsf{K}[t]^2 \Vert_{\rm op}
=\Vert \mathsf{K}[t] \Vert_{\rm op}^2
=1$, since $\mathsf{K}[t]$ is self-adjoint. 
Since $\mathsf{K}[t]$ is compact, $\mathsf{K}[t]^2$ is a linear compact operator. 
Therefore, $1$ or $-1$  is an eigenvalue of  $\mathsf{K}[t]^2$. 
Hence every complex number of absolute value one is an eigenvalue of $\mathsf{K}[t]$. 
This is a contradiction.   
\end{proof}

%
\subsection{A sufficient condition for (O7)}
%

We often easily find that the values of $\Phi(t,t)$ and $\Psi(t,t)$ 
for large negative $t$ 
as in the case of $u$ is an inner function or $u=\Gamma(\frac{1}{2}+iz)/\Gamma(\frac{1}{2}-iz)$. 
In such cases, the smoothness of $\Phi(t,x)$ and $\Psi(t,x)$ around the diagonal $x=t$ 
lead to the positive definiteness of $H(t)$ defined by \eqref{c_208} and \eqref{c_209}. 
In stating the following proposition, 
we refer to \cite[Definition 6.9]{MR0216289} for the values of distributions.

\begin{proposition} \label{prop_7_4}
Let $u \in U_{\rm loc}^1(\R)$. Suppose that (O1), (O2), (O3), (O4) are satisfied. 
Further, we suppose that there is an interval $I$ such that 
\begin{enumerate}
\item[(0)] $\Vert \mathsf{K}[t] \Vert_{\rm op}<1$ for $t \in I$; 
\item[(1)]  the derivatives $\frac{d}{dt} \Phi(t,t)$ and $\frac{d}{dt} \Psi(t,t)$ 
are defined as a distribution on $I$;  
\item[(2)]  the distributions $\frac{\partial}{\partial t} \Phi(t,x)$, 
$\frac{\partial}{\partial x} \Phi(t,x)$, 
$\frac{\partial}{\partial t} \Psi(t,x)$, 
$\frac{\partial}{\partial x} \Psi(t,x)$ for $x$ have values at $x=t$ 
for almost all $t \in I$; 
\item[(3)] all $t \mapsto \frac{\partial}{\partial t} \Phi(t,t)$, 
$t \mapsto \frac{\partial}{\partial x} \Phi(t,t)$, 
$t \mapsto \frac{\partial}{\partial t} \Psi(t,t)$, 
$t \mapsto \frac{\partial}{\partial x} \Psi(t,t)$ define 
distributions on $I$ and satisfy 
\begin{equation} \label{c_701}
\frac{d}{dt}\Phi(t,t)= 
\frac{\partial\Phi}{\partial t}(t,t) + \frac{\partial\Phi}{\partial x}(t,t), \quad 
\frac{d}{dt}\Psi(t,t)= 
\frac{\partial\Psi}{\partial t}(t,t) + \frac{\partial\Psi}{\partial x}(t,t). 
\end{equation}
\end{enumerate}
Then $\Re(\Phi(t,t)\overline{\Psi(t,t)})$ is a constant on $I$.  
\end{proposition}
\begin{proof}
Adding \eqref{c_503} and \eqref{c_506}, 
\[
\aligned 
\phi^++\phi^-
& =\frac{\Re(\Phi(t,t))}{\Re(\overline{\Psi(t,t)}\Phi(t,t))}
\left( \frac{\partial}{\partial t}\Psi(t,x) 
+ \frac{\partial}{\partial x}\Psi(t,x)
\right) \\
& \quad - i \, 
\frac{\Im(\Psi(t,t))}{\Re(\overline{\Psi(t,t)}\Phi(t,t))}
\left( 
\frac{\partial}{\partial t}\Phi(t,x) + \frac{\partial}{\partial x}\Phi(t,x) 
\right). 
\endaligned 
\]
On the other, adding \eqref{c_504} and \eqref{c_505},  
\[
\aligned 
\phi^++\phi^-
& = -\frac{\Re(\Psi(t,t))}{\Re(\overline{\Psi(t,t)}\Phi(t,t))}
\left( \frac{\partial}{\partial t}\Phi(t,x) 
+ \frac{\partial}{\partial x}\Phi(t,x)
\right) \\
& \quad + i \, 
\frac{\Im(\Phi(t,t))}{\Re(\overline{\Psi(t,t)}\Phi(t,t))}
\left( 
\frac{\partial}{\partial t}\Psi(t,x) + \frac{\partial}{\partial x}\Psi(t,x) 
\right). 
\endaligned 
\]
Therefore, 
\[
\aligned 
\frac{\overline{\Psi(t,t)}}{\Re(\overline{\Psi(t,t)}\Phi(t,t))}
& \left( 
\frac{\partial}{\partial t} + \frac{\partial}{\partial x}
\right)\Phi(t,x) \\
& +
\frac{\overline{\Phi(t,t)}}{\Re(\overline{\Psi(t,t)}\Phi(t,t))}
\left( \frac{\partial}{\partial t}
+ \frac{\partial}{\partial x}
\right) \Psi(t,x)
=0. 
\endaligned 
\]
Using this and \eqref{c_701},  we have 
\[
\aligned 
\frac{d}{dt} & \log \Re(\Phi(t,t)\overline{\Psi(t,t)}) \\
&=
\frac{\overline{\Phi(t,t)}}{\Re(\Phi(t,t)\overline{\Psi(t,t)})}\frac{d}{dt}\Psi(t,t)
+
 \frac{\overline{\Psi(t,t)}}{\Re(\Phi(t,t)\overline{\Psi(t,t)})}\frac{d}{dt}\Phi(t,t)
 \\
& \quad 
 +  
\overline{
\left(
\frac{\overline{\Phi(t,t)}}{\Re(\overline{\Psi(t,t)}\Phi(t,t))}\frac{d}{dt}\Psi(t,t)
+
\frac{\overline{\Psi(t,t)}}{\Re(\Phi(t,t)\overline{\Psi(t,t)})}\frac{d}{dt}\Phi(t,t)
\right)
} \\ 
& =0.
\endaligned 
\]
Hence $\Re(\Phi(t,t)\overline{\Psi(t,t)})$ is a constant. 
\end{proof}

%


\bigskip \noindent
\\
Department of Mathematics, 
School of Science, \\
Tokyo Institute of Technology \\
2-12-1 Ookayama, Meguro-ku, 
Tokyo 152-8551, JAPAN  \\
Email: {\tt msuzuki@math.titech.ac.jp}

\end{document}